\documentclass[11pt,onecolumn]{article}

\usepackage{a4wide}
\usepackage{amsmath,amsthm,epsfig,amssymb,amsbsy,mathtools}
\usepackage{enumerate}
\usepackage{comment}
\usepackage{algorithm}
\usepackage{algorithmic}
\usepackage{amsfonts}       
\usepackage{dsfont}
\usepackage{enumitem}
\usepackage{amssymb}

\usepackage{threeparttable}
\usepackage{chngcntr}
\usepackage{tikz}
\usepackage{tikz-3dplot}
\usepackage{pgfplots}
\usetikzlibrary{backgrounds}
\usepackage{makecell}
\usepackage{tabularx}
\usepackage{booktabs}
\usepackage{microtype}
\usepackage{subcaption}

\usepackage[giveninits=true,maxbibnames=9, maxcitenames=2, backend=biber]{biblatex}
\newenvironment{keywords}{\begin{paragraph}{Keywords:}
}
{
\end{paragraph}
}
    
\newenvironment{subclass}{\begin{paragraph}{AMS Subject Classification:}
}
{\end{paragraph}
}

\DeclareMathOperator*{\argmin}{arg\,min}

\DeclareMathOperator{\dom}{dom}
\DeclareMathOperator{\diag}{Diag}

\DeclareMathOperator{\intr}{int}
\DeclareMathOperator{\bdry}{bdry}

\DeclareMathOperator{\ran}{rge}

\DeclareMathOperator{\sign}{sign}

\DeclareMathOperator{\id}{id}
\DeclareMathOperator{\Diag}{Diag}
\DeclareMathOperator{\rank}{rank}
\DeclareMathOperator{\conv}{conv}
\DeclareMathOperator{\trace}{trace}

\newcommand{\bR}{\mathbb{R}}

\newcommand{\expec}{\mathbb{E}}

\newcommand{\bN}{\mathbb{N}}

\newcommand{\exR}{\overline{\mathbb{R}}}

\DeclareMathOperator{\gap}{gap}

\newcommand{\colfwd}{blue!70!black}
\newcommand{\colbck}{red!70!black}

\newcommand{\cC}{\mathcal{C}}

\newcommand{\cG}{\mathcal{G}}

\newcolumntype{Y}{>{\centering\arraybackslash}X}

\makeatletter
\newcommand{\prox}[3][\@nil]{%
  \def\tmp{#1}%
   \ifx\tmp\@nnil
       \operatorname{prox}_{#3}^{#2}
    \else
         \operatorname{prox}_{#3}^{#1 \star #2}
    \fi}

\newcommand{\bprox}[3][\@nil]{%
  \def\tmp{#1}%
   \ifx\tmp\@nnil
       \operatorname{bprox}_{#3}^{#2}
    \else
        \operatorname{bprox}_{#3}^{#1 #2}
    \fi}
\makeatother

\usepackage{hyperref}
\hypersetup{
    colorlinks=true,
    linkcolor=blue,
    filecolor=magenta,
    citecolor =magenta,  
    urlcolor=magenta,
    pdftitle={Constrained Stochastic Spectral Preconditioning}
    }

\usepackage[capitalize,nameinlink]{cleveref}[0.19]

\usepackage{aliascnt}

\crefname{section}{Section}{Sections}
\crefname{subsection}{Subsection}{Subsections}
\Crefname{section}{Section}{Sections}
\Crefname{subsection}{Subsection}{Subsections}

\Crefname{figure}{Figure}{Figures}

\crefformat{equation}{\textup{#2(#1)#3}}
\crefrangeformat{equation}{\textup{#3(#1)#4--#5(#2)#6}}
\crefmultiformat{equation}{\textup{#2(#1)#3}}{ and \textup{#2(#1)#3}}
{, \textup{#2(#1)#3}}{, and \textup{#2(#1)#3}}
\crefrangemultiformat{equation}{\textup{#3(#1)#4--#5(#2)#6}}%
{ and \textup{#3(#1)#4--#5(#2)#6}}{, \textup{#3(#1)#4--#5(#2)#6}}{, and \textup{#3(#1)#4--#5(#2)#6}}

\Crefformat{equation}{#2Equation~\textup{(#1)}#3}
\Crefrangeformat{equation}{Equations~\textup{#3(#1)#4--#5(#2)#6}}
\Crefmultiformat{equation}{Equations~\textup{#2(#1)#3}}{ and \textup{#2(#1)#3}}
{, \textup{#2(#1)#3}}{, and \textup{#2(#1)#3}}
\Crefrangemultiformat{equation}{Equations~\textup{#3(#1)#4--#5(#2)#6}}%
{ and \textup{#3(#1)#4--#5(#2)#6}}{, \textup{#3(#1)#4--#5(#2)#6}}{, and \textup{#3(#1)#4--#5(#2)#6}}

\newtheorem{theorem}{Theorem}[section]

\newaliascnt{lemma}{theorem}
\newtheorem{lemma}[lemma]{Lemma}
\aliascntresetthe{lemma}

\newaliascnt{proposition}{theorem}
\newtheorem{proposition}[proposition]{Proposition}
\aliascntresetthe{proposition}

\newaliascnt{corollary}{theorem}
\newtheorem{corollary}[corollary]{Corollary}
\aliascntresetthe{corollary}

\newaliascnt{assumption}{theorem}
\newtheorem{assumption}[assumption]{Assumption}
\aliascntresetthe{assumption}

\newaliascnt{definition}{theorem}
\newtheorem{definition}[definition]{Definition}
\aliascntresetthe{definition}

\newaliascnt{fact}{theorem}

\aliascntresetthe{fact}

\newtheoremstyle{boldremark}
    {\dimexpr\topsep/2\relax} 
    {\dimexpr\topsep/2\relax} 
    {}          
    {}          
    {\bfseries} 
    {.}         
    {.5em}      
    {}          

\theoremstyle{boldremark}
\newaliascnt{remark}{theorem}
\newtheorem{remark}[remark]{Remark}
\aliascntresetthe{remark}

\theoremstyle{boldremark}
\newaliascnt{example}{theorem}
\newtheorem{example}[example]{Example}
\aliascntresetthe{example}

\crefname{theorem}{Theorem}{Theorems}
\Crefname{theorem}{Theorem}{Theorems}
\crefname{lemma}{Lemma}{Lemmas}
\Crefname{lemma}{Lemma}{Lemmas}
\crefname{proposition}{Proposition}{Propositions}
\Crefname{proposition}{Proposition}{Propositions}
\crefname{corollary}{Corollary}{Corollaries}
\Crefname{corollary}{Corollary}{Corollaries}
\crefname{definition}{Definition}{Definitions}
\Crefname{definition}{Definition}{Definitions}
\crefname{fact}{Fact}{Facts}
\Crefname{fact}{Fact}{Facts}
\crefname{remark}{Remark}{Remarks}
\Crefname{remark}{Remark}{Remarks}
\crefname{example}{Example}{Examples}
\Crefname{example}{Example}{Examples}
\crefname{assumption}{Assumption}{Assumptions}
\Crefname{assumption}{Assumption}{Assumptions}

\newlist{lemenum}{enumerate}{1} 
\setlist[lemenum]{label=(\roman*), ref=\thelemma(\roman*), font=\rm}
\crefalias{lemenumi}{lemma}

\newlist{propenum}{enumerate}{1} 
\setlist[propenum]{label=(\roman*), ref=\theproposition(\roman*), font=\rm}
\crefalias{propenumi}{proposition}

\newlist{thmenum}{enumerate}{1} 
\setlist[thmenum]{label=(\roman*), ref=\thetheorem(\roman*), font=\rm}
\crefalias{thmenumi}{theorem}
\counterwithin*{thmenumi}{theorem}

\newlist{defenum}{enumerate}{1} 
\setlist[defenum]{label=(\roman*), ref=\thedefinition(\roman*), font=\rm}
\crefalias{defenumi}{definition}

\newlist{corenum}{enumerate}{1} 
\setlist[corenum]{label=(\roman*), ref=\thecorollary(\roman*), font=\rm} 
\crefalias{corenumi}{corollary}

\newlist{examplenum}{enumerate}{1} 
\setlist[examplenum]{label=(\roman*), ref=\theexample(\roman*), font=\rm} 
\crefalias{examplenumi}{example}

\usepackage{xcolor}

    \makeatletter
\def\@fnsymbol#1{\ensuremath{\ifcase#1\or \dagger\or \ddagger\or
   \mathsection\or \mathparagraph\or \|\or **\or \dagger\dagger
   \or \ddagger\ddagger \else\@ctrerr\fi}}
    \makeatother

\title{Constrained Stochastic Spectral Preconditioning Converges for Nonconvex Objectives}

\author{Konstantinos Oikonomidis$^*$\thanks{KU Leuven,
		Department of Electrical Engineering, STADIUS Center for Dynamical Systems, Signal Processing, and Data Analytics, 
		{\tt%
            \href{mailto:konstantinos.oikonomidis@esat.kuleuven.be}{\{konstantinos.oikonomidis,}%
            \href{mailto:jan.quan@esat.kuleuven.be}{jan.quan,}%
			\href{mailto:panos.patrinos@esat.kuleuven.be}{panos.patrinos\}}%
			\href{mailto:konstantinos.oikonomidis@esat.kuleuven.be,jan.quan@esat.kuleuven.be,panos.patrinos@esat.kuleuven.be}{@esat.kuleuven.be}%
		}
	} \and Jan Quan$^*$\footnotemark[1] \and Kimon Antonakopoulos\thanks{\'Ecole Polytechnique F\'ed\'erale de Lausanne (EPFL), Laboratory for Information and Inference Systems (LIONS)~
    {\tt%
    \href{mailto:kimon.antonakopoulos@epfl.ch}{\{kimon.antonakopoulos,}%
    \href{mailto:volkan.cevher@epfl.ch}{volkan.cevher\}}%
    \href{mailto:kimon.antonakopoulos@epfl.ch,volkan.cevher@epfl.ch}
    {@epfl.ch}%
    }
    }
    \and Antonio Silveti-Falls\thanks{CVN, CentraleSup\'elec, Universit\'e Paris-Saclay, Inria
    {\tt%
    \href{mailto:tonys.falls@gmail.com}{tonys.falls@gmail.com}
    }
    }
    \and Volkan Cevher$^\P$\footnotemark[2]
    \and Panagiotis Patrinos$^\P$\footnotemark[1]}

\pgfplotsset{compat=1.18}
\begin{document}

\maketitle

\def\thefootnote{*}\footnotetext{Equal contribution}\def\thefootnote{\arabic{footnote}}
\def\thefootnote{$\P$}\footnotetext{Equal supervision}\def\thefootnote{\arabic{footnote}}

\begin{abstract}
    In this work, we develop proximal preconditioned gradient methods with a focus on spectral gradient methods providing a proximal extension to the Muon and Scion optimizers. We introduce a family of stochastic algorithms that can handle a wide variety of convex and nonconvex constraints and study its convergence under heavy-tailed noise, through a novel analysis tailored to the geometry of the proposed methods. We further propose a variance-reduced version, which achieves faster convergence under standard noise assumptions. Finally, we show that the polynomial iterations used in Muon are more accurately captured by a nonlinear preconditioner than by the ideal matrix sign, leading to a convergence analysis that more faithfully reflects practical implementations.
\end{abstract}
\begin{subclass}
65K05 $\cdot$ 49J52 $\cdot$ 90C30
\end{subclass}
\tableofcontents

\section{Introduction}
Stochastic spectral gradient methods \cite{carlson2015preconditioned,jordan2024muon,pethick2025trainingdeeplearningmodels} have received widespread attention recently due to their efficiency in training deep neural networks. Unlike classical stochastic gradient methods, which rely directly on noisy gradient estimates, spectral methods apply an approximation of the matrix sign \cite{amsel2025polar}, fundamentally altering the geometry of the algorithm and leading to improved performance. More precisely, if $D$ is an estimate of the gradient and $D = U \Sigma V^\top$ a reduced SVD, only the product $U V^\top$ is considered in the update.

This modification of the update direction places spectral gradient methods within a broader class of adaptive algorithms that normalize the gradients before updating the variables. Prominent examples include gradient clipping \cite{zhang2020gradientclippingacceleratestraining,koloskova2023revisiting}, normalized gradient \cite{cutkosky2020momentum} and sign gradient \cite{carlson2015preconditioned,bernstein2018signsgd} among others. These methods have been studied extensively not only for the practical benefit of stabilizing the training procedure, but also for their theoretical properties of adapting to generalized forms of smoothness \cite{zhang2020gradientclippingacceleratestraining,vankov2024optimizing} and handling heavy-tailed noise \cite{cutkosky2020momentum,hubler2024gradient}.

Despite this growing body of work, existing analyses have predominantly focused on unconstrained optimization involving a single smooth objective. As a result, comparatively little is known about the behavior of stochastic spectral and related adaptive gradient methods in the presence of constraints or explicit regularization. The few works that address such settings are restricted to convex nonsmooth terms \cite{chen2023lion,pmlr-v235-gorbunov24a,liu2024communication,pethick2025trainingdeeplearningmodels,kovalev2025understanding,chen2025muon}, most commonly indicator functions of closed, convex sets. Moreover, extending spectral or sign-based updates to constrained problems is not straightforward since naive application of standard proximal operators can result in fixed points that are unrelated to the original optimization objective~\cite{pmlr-v235-gorbunov24a}.

However, regularization and constraints play an important role in machine learning practice. Except for standard practices like $\ell_1$ and $\ell_2$ regularization \cite{shalev2014understanding} that are linked to the popular weight decay mechanism, recent works consider putting hard constraints on the weights.  Prominent examples include constraints based on the spectral norm \cite{newhouse2025training,xie2026controlled} and the Euclidean norm \cite{karras2024analyzing,loshchilov2024ngpt}, which restrict weights to the corresponding norm ball or its boundary, while empirical evidence suggests that in many neural network training tasks norm control improves performance \cite{pethick2025trainingdeeplearningmodels}. In this work we therefore aim at bridging this gap by introducing proximal preconditioned gradient methods that can naturally tackle composite minimization problems.

\paragraph{Contributions}
More concretely, our contributions can be summarized as follows.
\begin{itemize}
    \item In the deterministic setting, we extend the analysis of \cite{laude2025anisotropic} which leads to a preconditioned proximal gradient method for problems where the smooth term satisfies a relaxed smoothness assumption and the nonsmooth term is potentially nonconvex. Our analysis can be found in \cref{app:det}. We moreover extend to the matrix setup where we use tools from convex analysis to obtain closed forms of the generalized prox operators we study in \cref{sec:matrix_ext}. Some important backward steps are summarized in \cref{tab:back_step_euc}, which practitioners can utilize to explore different combinations of algorithms and constraints.

    \item We propose a stochastic proximal preconditioned gradient algorithm for composite minimization problems where the nonsmooth term can be nonconvex, including thus the indicators of interesting constraint sets such as the spectral sphere and the Stiefel manifold. We study the convergence of the method under heavy-tailed noise and obtain improved convergence guarantees using variance-reduction techniques under bounded variance in \cref{sec:stoch}. To the best of our knowledge, this is the first work on stochastic methods for problems with general nonsmooth terms $g$ and under heavy-tailed noise. Moreover, our analysis is unifying, covering a whole family of algorithms, and the proposed methods are adaptive, requiring no prior knowledge of smoothness or noise constants.
    
    \item As a separate contribution, we observe that nonlinear preconditioning more faithfully captures the spectral gradient update arising in polynomial approximations (e.g., Newton--Schulz, Polar Express) of the matrix sign function, see \cref{fig:pe_approx}. We analyze the properties in both deterministic and stochastic regimes through a novel analysis in \cref{subsec:inexact_main_text} and \cref{app:polar_full}. Notably, we are also able to capture the pre-normalization step prior to the polynomial iteration and thus more accurately describe the convergence of Muon and related spectral algorithms.
\end{itemize}

\begin{figure}[hbpt]
    \centering
    \begin{subfigure}[b]{0.48\textwidth}
        \centering
        \includegraphics[width=0.8\linewidth]{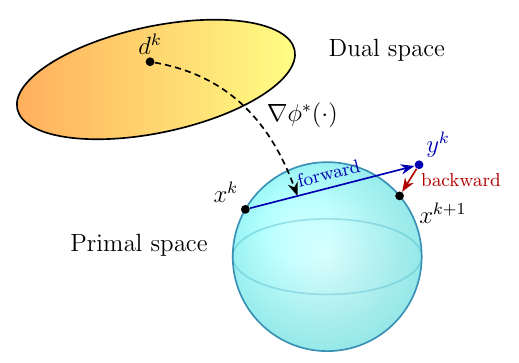}
        \caption{Visualization algorithm.}
        \label{fig:fwd_bck_vis}
    \end{subfigure}
    \hfill
    \begin{subfigure}[b]{0.48\textwidth}
        \centering
        \includegraphics[width=0.8\linewidth]{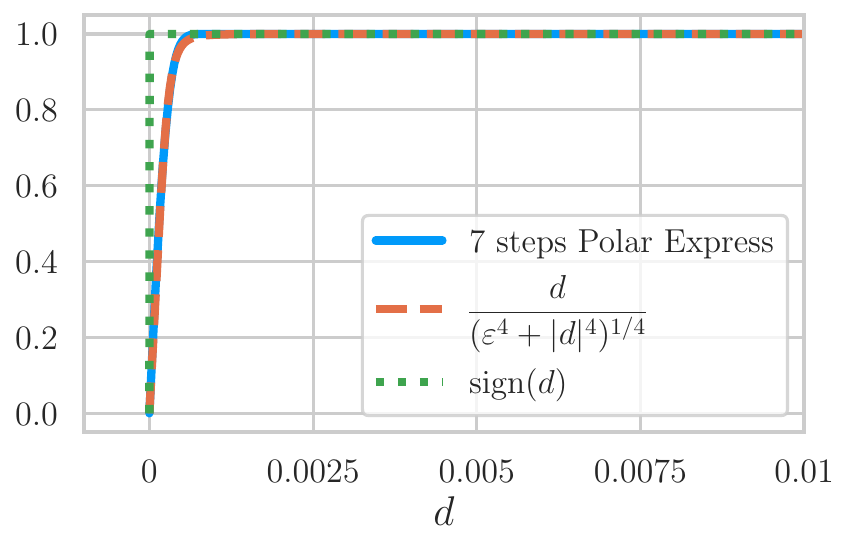}
        \caption{Polar Express and nonlinear preconditioner.}
        \label{fig:pe_approx}
    \end{subfigure}
    
    \caption{(a) Diagram of the proximal preconditioned gradient method as described in \eqref{eq:update}. The forward step applies a nonlinear preconditioner $\nabla \phi^*$ to the dual space direction $d^k$ (e.g., the gradient $\nabla f(x^k)$) to compute the forward step $y^k$ in the primal space. Since this might leave the constraint set $C$, a backward step is performed to compute $x^{k+1}$. (b) The Polar Express \cite{amsel2025polar} fits more closely to $\frac{d}{(\varepsilon^4 + |d|^4)^{1/4}}$ with $\varepsilon=3\times 10^{-4}$ than to the actual sign, showing that our theory more faithfully captures the practical implementations of Muon.}
    \label{fig:merged_figure}
\end{figure}

\subsection{Related work}
\paragraph{Spectral gradient and norm-constrained lmos} The algorithms that we introduce in this paper are mostly related to the layer-wise family of algorithms proposed in recent works~\cite{jordan2024muon,pethick2025trainingdeeplearningmodels}. The standard theoretical analysis for this class of methods was performed in \cite{pethick2025trainingdeeplearningmodels}, based on the linear minimization oracle over a (user-defined) norm ball. Moreover, constrained optimization problems where the constraint sets match those of the linear minimization oracle were studied via a Frank--Wolfe scheme. The literature on this type of methods has since been vast and fast expanding and so we cover only a subset of the relevant works here. 

In \cite{pethick2025generalizedgradientnormclipping} the algorithmic family was extended using a hybrid between the conditional gradient and steepest descent algorithm to tackle problems under generalized smoothness, and in \cite{kovalev2025understanding}, a non-Euclidean trust region perspective was employed that also leads to faster convergence under star-convexity of the objective function. A preconditioning and steepest descent point of view was given in \cite{lau2025polargrad}, and in \cite{chen2025muon}, Muon (with decoupled weight decay) was connected to the LION-$\mathcal{K}$ family of algorithms and thus to an optimization problem with constraints on the spectral norm of the parameters, following the more general analysis of \cite{chen2023lion}. A stochastic Frank--Wolfe analysis for the convergence of Muon with weight decay was also performed in~\cite{sfyraki2025lions}.

Further extensions to the setting where the cost function is layer-wise $(L_0,L_1)$-smooth were provided in \cite{riabinin2025gluon}, while \cite{zhang2025adagrad} introduced an adaptive method that combines AdaGrad with Muon. In \cite{liu2025mars}, the MARS \cite{yuan2024mars} variance-reduction technique was utilized on top of the Muon algorithm leading to improved convergence rates at the cost of additional samples per iteration and a similar approach was considered in \cite{qian2025muon}. We remark that significant attention has also been given to improving the computation of the generalized matrix sign operation employed by spectral gradient methods \cite{cesistasqueezing,ahn2025dion}, with \cite{amsel2025polar} using an optimal polynomial approximation in some sense, that leads to improved performance when used in Muon. Furthermore, various works have been devoted to analyzing the convergence of Muon under inexactness, in an effort to better describe what is happening in practice~\cite{shulgin2025beyond}.

\paragraph{Normalized gradient methods}
The methods we analyze are tightly connected to the normalized gradient and gradient clipping literature. Having found widespread use in neural network training, these methods have also received interest due to their significant theoretical properties. Normalizing (or clipping) the stochastic gradient leads to tighter control on the trajectory of the algorithm since the distance (or an upper bound of the distance) between two consecutive iterates depends solely on the stepsize and thus greatly simplifies the analysis under more aggressive forms of noise. This approach has been utilized in \cite{cutkosky2021high} where normalization was combined with clipped stochastic gradient estimates to prove high probability bounds under heavy-tailed noise and in \cite{hubler2024gradient,liu2024nonconvex} where the requirement of clipped gradients in the momentum estimate was lifted. An extension of this approach to the sign gradient method \cite{carlson2015preconditioned,bernstein2018signsgd} was considered in \cite{kornilov2025sign}. Nevertheless, recent results \cite{fatkhullin2025can} seem to indicate that vanilla SGD converges under heavy-tailed noise. We remark that \cite{pmlr-v235-gorbunov24a} also considers composite optimization problems, though both the smooth and nonsmooth term are assumed to satisfy some form of convexity.

Gradient normalization and its many variants have also been connected to tackling problems under generalized smoothness assumptions. Starting with the introduction of $(L_0,L_1)$-smoothness in \cite{zhang2020gradientclippingacceleratestraining} in an effort to better describe the smoothness conditions of neural networks, a large body of work has been devoted to this line of research \cite{zhang2020improved,koloskova2023revisiting,wang2023convergence,li2023convex,vankov2024optimizing}. This idea has been expanded upon, with other notions of generalized smoothness and adaptive algorithms appearing in \cite{chen2023generalized,tyurin2024toward}. A unifying nonlinear preconditioning perspective, which associated this family of methods with the anisotropic smoothness condition from \cite{laude2025anisotropic} was presented in \cite{oikonomidis2025nonlinearlypreconditionedgradientmethods}. Further related work to the method we analyze is provided in \cref{sec:alg}.

\subsection{Notation}
Throughout the paper $E$ is a Euclidean space with inner product $\langle\cdot,\cdot\rangle$ and induced norm $\|\cdot\|$. We denote by $\mathcal{C}^k(Y)$ the class of functions which are $k$ times continuously differentiable on an open set $Y \subseteq E$. We denote $\exR = \bR \cup\{-\infty, +\infty\}$ and say that $f:E \to \exR$ is proper if $f$ is not identically $+\infty$ and $f(x) > -\infty$ for all $x \in E$. The indicator function of a set $C \subseteq E$ is defined as $\delta_C(x) = 0$ if $x \in C$ and $+\infty$ otherwise. For a proper function $f:E \to \exR$ and $\lambda\geq 0$ we define the episcaling $(\lambda \star f)(x) = \lambda f(\lambda^{-1}x)$ for $\lambda > 0$ and $(\lambda \star f)(x) = \delta_{\{0\}}(x)$ otherwise. The convex conjugate is defined as $f^*(y) = \sup_{x \in E}\langle y,x \rangle - f(x)$ with $(\lambda \star f)^*=\lambda f^*$. A function $f:E \to \exR$ is prox-bounded if there exists $\gamma > 0$ such that $\inf_{x \in E}f(x) + \tfrac{1}{2\gamma}\|x-\bar x\|^2 > -\infty$ for some $\bar x \in E$. For a strictly convex and $\cC^1(E)$ function $f$ we denote its Bregman divergence $D_f(x,\bar x) = f(x)-f(\bar x) - \langle \nabla f(\bar x),x-\bar x \rangle$ for $x,\bar x \in E$. The real orthogonal group in dimension $n$ is denoted as $O(n) = \{A \in \bR^{n\times n} \mid A^\top A = AA^\top = I_n\}$, with $I_n$ the identity matrix of size $n\times n$. A function $F:\bR^{m\times n} \to \exR$ is real orthogonal invariant if $F(UXV) = F(X)$ for all $X\in\bR^{m\times n}, U\in O(m), V\in O(n)$. The singular value mapping $\sigma : \bR^{m\times n} \to \bR^{\min(m,n)}$ maps a matrix $X\in\bR^{m\times n}$ to its vector of singular values in nonincreasing order. We use $\Diag x$ to denote a (possibly rectangular) diagonal matrix with the vector $x$ on its main diagonal, where the dimensions are implied by the context. A function $f:\bR^n\to \exR$ is absolutely symmetric if $f(x)=f(|x|^\downarrow)$ for all $x\in\bR^n$, where $|x|^\downarrow$ denotes the vector with entries $|x_i|$ in nonincreasing order.

\section{Problem formulation and main algorithm} \label{sec:alg}
In the following we formulate the problem that we tackle in this paper
\begin{equation} \label{eq:problem} \tag{P}
    \min_{x \in E}\ F(x) := f(x) + g(x),
\end{equation}
where $f \in \cC^1(E)$ is a generally nonconvex function and $g:E \to \exR$ is a possibly nonsmooth and nonconvex regularizer. Throughout the paper we assume that $F_\star = \inf F > -\infty$, i.e., the objective is lower bounded. When we have access to the full gradient $\nabla f(x)$ we refer to \eqref{eq:problem} as the deterministic setting, whereas access only to a stochastic estimate $\nabla f(x,\xi)$ of $\nabla f(x)$ corresponds to the stochastic setting. We are mainly interested in the case where $g$ encodes constraints, i.e., it is the indicator of some closed set. Our main tool to tackle \eqref{eq:problem} is a generalization of the anisotropic proximal gradient method from \cite{laude2025anisotropic}, which takes the following form: starting from $x^0 \in \dom g$, compute
\begin{equation} \label{eq:update}
    \begin{aligned}
        \textcolor{\colfwd}{y^k} & \textcolor{\colfwd}{= x^k - \gamma_k \nabla \phi^*(d^k)}
        \\
        \textcolor{\colbck}{x^{k+1}} & \textcolor{\colbck}{\in \argmin_{x \in E}\; g(x) +( \gamma_k \star \phi)(x-y^k)},
    \end{aligned}
\end{equation}
for some direction $d^k \in E$ that carries information on $\nabla f(x^k)$, i.e., a (stochastic) gradient or momentum estimate. In a classical manner, we refer to the two steps as the \textcolor{\colfwd}{\emph{forward step}} and the \textcolor{\colbck}{\emph{backward}} or \textcolor{\colbck}{\emph{proximal step}} respectively. The function $\phi:E \to \exR$ plays the role of the \emph{reference function} and its convex conjugate $\phi^*$ that of the \emph{dual reference function}. This iteration is visualized in \cref{fig:fwd_bck_vis}. We next formulate our assumptions on $\phi$, which we consider valid throughout the remainder of the paper.
\begin{assumption} \label[assumption]{assum:sc}
    The reference function $\phi:E \to \exR$ is lsc, $\mu$-strongly convex and even with $\phi \geq 0$ and $\phi(0) = 0$. $\intr \dom \phi \neq \emptyset$; $\phi \in \cC^1(\intr \dom \phi)$ and for any sequence $\{x^k\}_{k \in \bN_0}$ in $\intr \dom \phi$ that converges to some boundary point of $ \dom \phi$, $\|\nabla \phi(x^k)\| \to +\infty$.
\end{assumption}
Note that $\phi^*$ is strictly convex and $\nabla\phi^*$ is odd by \cref{thm:conv_conj}. \Cref{assum:sc} thus serves two purposes: on the one hand since $\phi$ is strongly convex the minimization problem in \eqref{eq:update} is meaningful even when $g$ is nonconvex and on the other hand, it implies that the fixed points of the algorithm correspond to stationary points of the original problem. To better see that, note that for $d^k=\nabla f(x^k)$, using the optimality conditions for the subproblem, $x^{k+1}=x^k$ implies that $-\nabla \phi^*(\nabla f(x^{k+1})) = \nabla \phi^*(\widetilde \nabla g(x^{k+1}))$ for some $\widetilde \nabla g(x^{k+1}) \in \partial g(x^{k+1})$ and thus $-\nabla f(x^{k+1}) = \widetilde \nabla g(x^{k+1})$ due to the strict monotonicity of $\nabla \phi^*$.

\begin{figure}[tbp]
    \centering
    \begin{subfigure}[t]{0.24\linewidth}
        \centering
        \includegraphics[trim=0cm 0cm 11cm 0cm, clip, width=\linewidth]{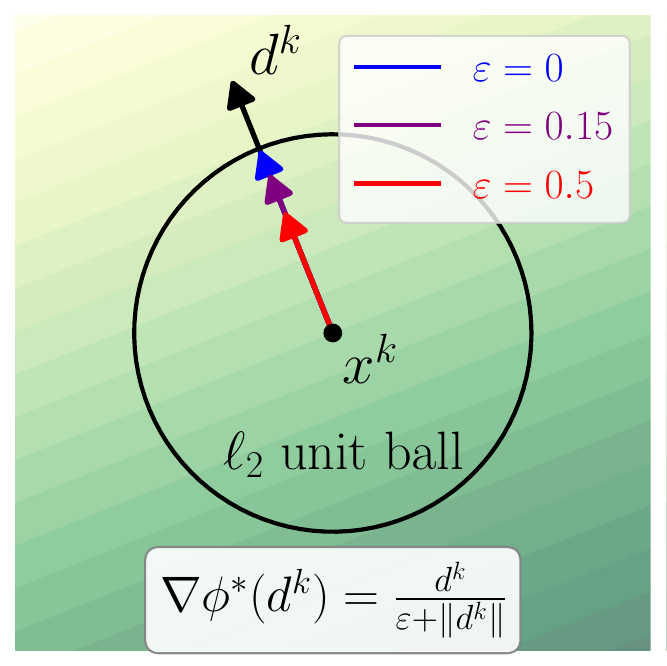}
        \caption{}
        \label{fig:smooth_lmo_left}
    \end{subfigure}
    \hfill
    \begin{subfigure}[t]{0.24\linewidth}
        \centering
        \includegraphics[trim=11cm 0cm 0cm 0cm, clip, width=\linewidth]{figs/smooth_lmo.pdf}
        \caption{}
        \label{fig:smooth_lmo_right}
    \end{subfigure}
    \hfill
    \begin{subfigure}[t]{0.24\linewidth}
        \centering
        \includegraphics[width=\linewidth, page=1]{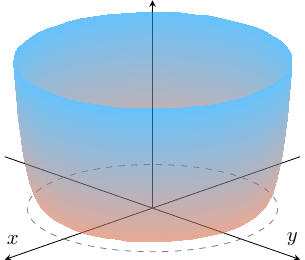}
        \caption{}
        \label{fig:barrier_iso}
    \end{subfigure}
    \hfill
    \begin{subfigure}[t]{0.24\linewidth}
        \centering
        \includegraphics[width=\linewidth, page=2]{figs/barrier.pdf}
        \caption{}
        \label{fig:barrier_aniso}
    \end{subfigure}
    
    \caption{(a-b) Illustration of the nonlinear preconditioning induced by $\phi$ as described in Example~\ref{example:norm_grad}. As $\varepsilon\to 0^+$, the linear minimization oracle over the unit ball is recovered. (c-d) Visualization of the reference functions $\phi$ in \cref{example:norm_grad} as barriers. Both isotropic (c) and anisotropic (d) explode towards the boundary of their domain thus enforcing the corresponding constraints. In contrast to the indicator function that is identically zero in the set, the barrier also acts as regularization.}
    \label{fig:combined_all}
\end{figure}

\begin{remark} \label{rem:stationarity}
    The preceding discussion also highlights why the standard Euclidean prox is not a good choice for the backward step. Indeed the fixed points of the algorithm would satisfy $-\nabla \phi^*(\nabla f(x^{k+1})) \in \partial g(x^{k+1})$ and thus in general would not imply the stationarity of $F$, except for specific cases. One such a case is when $\phi = h \circ \|\cdot\|$ for some $h$ satisfying \cref{assum:sc} and $g$ is the indicator of some closed set, where the inclusion takes the form $-\lambda \nabla f(x^{k+1}) \in \partial g(x^{k+1})$ for some constant $\lambda \geq 0$. This then implies $-\nabla f(x^{k+1}) \in \partial g(x^{k+1})$ due to the conic structure of the subdifferential of $g$ \cite[Proposition 1.79]{Mordukhovich2006}.
\end{remark}

It has been shown \cite{jakovetic2023nonlinear,chen2023lion,oikonomidis2025nonlinearly} that (the unconstrained version of) the update described in~\eqref{eq:update} covers a wide variety of algorithms from the normalized gradient literature.
\begin{example} \label{example:norm_grad}
    Let $\phi(x) = -\varepsilon(\ln(1-\|x\|)+\|x\|)$. Then, we have that 
    \begin{equation*}
        x^{k+1} = x^k - \gamma_k d^k / (\varepsilon + \|d^k\|).
    \end{equation*}
    It is straightforward thus that the scheme recovers the smooth version of the normalized gradient method \cite[Equation (6)]{zhang2020gradientclippingacceleratestraining}, while as $\varepsilon \to 0^+$, it approximates the standard normalized gradient method. On the other hand, by choosing $\phi(x) = -\varepsilon \sum_{i=1}^n (\ln(1-|x_i|)+|x_i|)$ we are led to an algorithm where each coordinate is updated as
    \begin{equation*}
        x^{k+1}_i = x^k_i - \gamma_k d^k_i/(\varepsilon + |d^k_i|)
    \end{equation*}
    and we thus approximate the sign gradient method. \Cref{fig:smooth_lmo_left,fig:smooth_lmo_right} visualize the effect of the smoothing parameter $\varepsilon$ for both of these reference functions.
\end{example}

As we explain in the following remarks, the algorithm is closely related to the norm-constrained lmo methods and actually provides a proximal viewpoint on the use of decoupled weight decay in adaptive methods.
\begin{remark}[Connection with unconstrained SCG \cite{pethick2025trainingdeeplearningmodels}]
    The forward step $x^k - \gamma \nabla \phi^*(\nabla f(x^k))$ indicates also the relation of the method with the norm-constrained lmo methods \cite{pethick2025trainingdeeplearningmodels,pethick2025generalizedgradientnormclipping} and more precisely the unconstrained SCG algorithm, where a linear minimization oracle is used, $\text{lmo}(\nabla f(x^k)) \in \argmin_{x \in C} \langle \nabla f(x^k),x \rangle$ and $C$ is some convex, compact set. It is then straightforward that $\text{lmo}(\nabla f(x^k)) \in \partial \phi^*(-\nabla f(x^k))$ where $\phi = \delta_C$ the indicator of $C$ and thus $\phi^*$ is its support function \cite[Example 13.3(i)]{bauschke2017correction}. When $C$ is some unit norm ball, the reference functions $\phi$ from \cref{example:norm_grad} can be viewed as barrier functions enforcing respectively $\ell_2$ and $\ell_\infty$ constraints while keeping desired properties such as strong convexity, thus clarifying the connection between the two methods. 
\end{remark}
\begin{remark}[Connection with SCG \cite{pethick2025trainingdeeplearningmodels}] \label{rem:scion_conn}
    In \cite{pethick2025trainingdeeplearningmodels}, a method for constrained optimization over convex compact sets, based on the Frank--Wolfe scheme, was also introduced. The main iterate there takes the form
    \begin{equation*}
        x^{k+1} = (1-\gamma_k)x^k + \gamma_k \text{lmo}(d^k),
    \end{equation*}
    where $\gamma_k$ plays the role of the stepsize. This algorithm can approximately be seen as an instance of \eqref{eq:update} when $\phi$ is a barrier function as in the previous remark and $g = \tfrac{1}{\lambda} \star \phi$. From the optimality conditions of the proximal step in \eqref{eq:update} it is then clear that 
    \[
        x^{k+1} = \left(1-\tfrac{\gamma_k \lambda}{1+\gamma_k \lambda}\right)x^k - \tfrac{\gamma_k}{1+\gamma_k \lambda} \nabla \phi^*(d^k)
    \]
    thus leading to SCG with a suitably chosen stepsize. In fact, minimizing \(f + \tfrac{1}{\lambda} \star \phi\) corresponds to an approximation of the barrier function method \cite[\S 1.3]{nesterov2004introductory} and we obtain the solution of the original constrained problem as $\varepsilon \to 0^+$. An illustration of $\phi$ acting as a barrier function is shown in \cref{fig:barrier_iso,fig:barrier_aniso}. The fact that adding weight decay to the Muon optimizer leads to the solution of the ball-constrained optimization problem was studied extensively in \cite{chen2025muon} where $\phi$ is the indicator of the spectral ball.
\end{remark}

\paragraph{Dual space preconditioning and anisotropic proximal gradient}
The forward step in \eqref{eq:update} was introduced in \cite{maddison2021dual} for smooth convex minimization problems under a generalized cocoercivity condition and the proximal step was studied in \cite{laude2019optimization} for general nonconvex composite problems. In fact, the approximate barrier functions described in \cref{example:norm_grad} can be found in \cite[Equation (51)]{maddison2021dual} and \cite[Table 2]{laude2019optimization}. In \cite{laude2023dualities}, it was shown that the forward step corresponds to a majorization-minimization procedure for a new class of functions that was then extended in \cite{laude2025anisotropic} into a full convergence analysis of the method described in \eqref{eq:update} in the deterministic setting, for $\phi$ with full domain. There, it was analyzed under a condition called anisotropic smoothness that generalized Euclidean Lipschitz smoothness, inspired by the theory of $\Phi$-convexity \cite{dolecki1978convexity,Vil08}.

The forward step of \eqref{eq:update} was then shown in \cite{chen2023lion} to be connected to the Lion \cite{chen2023symbolic} optimizer and a Lyapunov analysis in both continuous and discrete time was performed, showing that weight decay leads to the solution of a constrained optimization problem. A similar approach was considered then in \cite{chen2025muon} for the Muon algorithm under spectral ball constraints, relying on increasing batch sizes in the stochastic setting. In \cite{oikonomidis2025nonlinearly,bodard2025escapingsaddlepointslipschitz} it was shown that the forward step of \eqref{eq:update} can actually tackle problems that are beyond Lipschitz smoothness and in fact even $(L_0,L_1)$-smoothness.

Note that the approaches of \cite{chen2023lion,pethick2025trainingdeeplearningmodels} couple the nonsmooth term with the preconditioner $\nabla \phi^*$, i.e., in our notation $g \equiv \phi$ ($\mathcal{K}^*$ in \cite{chen2023lion} and $\delta_D$ where $D$ is the constraint set in    \cite[Theorem 5.6]{pethick2025trainingdeeplearningmodels}), and thus cannot be generalized to nonconvex $g$, or in fact to any nonsmooth term that is not coupled with the update of the algorithm. For that reason we build upon the approach of \cite{laude2025anisotropic} that leads to effective proximal algorithms for general $g$, while also handling relaxed smoothness conditions on $f$. Nevertheless, the analysis in \cite{laude2025anisotropic} is limited to reference functions with full domain and thus does not immediately cover the setting of this paper. To that aim we extend the analysis in \cref{app:det}.

\paragraph{The proximal step}
The backward step of the main iteration \eqref{eq:update} might seem unintuitive upon first inspection. Nevertheless, it is a well-studied object known as the anisotropic proximity operator \cite[Definition 2.10]{combettes2013moreau} or anisotropic proximal mapping \cite[Definition 3.7]{laude2025anisotropic}. It is straightforward that for $\phi = \tfrac{1}{2}\|\cdot\|^2$ it corresponds to the standard Euclidean proximal operator. In fact, it admits a generalization of the classical Moreau's decomposition \cite[Theorem 3.2]{teboulle1992entropic} and \cite[Proposition 1.3]{combettes2013moreau} which involves the Bregman proximal mapping \cite{bauschke2017descent} that is standard in the related literature.

In many interesting cases, computing the backward step in \eqref{eq:update} corresponds to computing (some nonlinear transformation of) the standard Euclidean prox, thus reducing the overhead. Following \cite{oikonomidis2025nonlinearlypreconditionedgradientmethods} to ease presentation, we say that $\phi$ is \emph{isotropic} if $\phi = h \circ \|\cdot\|$ for some $h : \bR \to \exR$ satisfying \cref{assum:sc} and \emph{anisotropic} or separable if $\phi(x) = \sum_{i=1}^n h_i(x)$. 
\begin{example} \label{ex:isotropic_proj}
    Let $\phi$ be isotropic and $C$ a closed set. Then,
    \begin{equation*}
        \argmin_{x \in E}\; \delta_C(x) + (\gamma \star \phi)(x-y^k) = P_C(y^k),
    \end{equation*}
    i.e., the backward step is the standard Euclidean projection of $y^k$ onto $C$. This follows by noting that $\phi$ is simply a strictly increasing function of the Euclidean norm and is thus minimized at the same point. We present several examples of constraint sets and corresponding reference functions where the proximal step is easy to compute in \cref{tab:back_step_euc}.
    
    It is important to stress that in our analysis we only assume prox-boundedness of $g$ thus covering a wide variety of nonsmooth functions. Prox-bounded functions include convex ones \cite[Theorem 2.26]{RoWe98} and in general functions that are bounded from below \cite[Exercise 1.24]{RoWe98} thus including the indicators of closed but possibly nonconvex sets. Therefore, the proposed methods can tackle interesting constraint sets such as the spectral sphere and Stiefel manifold that have recently received attention for their applications in deep learning \cite{miyato2018spectral,newhouse2025training,xie2026controlled}.
\end{example}

\begin{table*}
    \centering
    \caption{\label{tab:back_step_euc} Examples of the backward step described in \eqref{eq:update} for various constraint sets $C$. The first half is associated to anisotropic reference functions, while the second half assumes spectral anisotropic reference functions as in \Cref{cor:spectral_ref}. Their isotropic counterparts always reduce to the Euclidean projection in light of \Cref{ex:isotropic_proj}. We moreover assume $n \leq m$ without loss of generality. $U\in O(m)$ and $V\in O(n)$ denote real orthogonal matrices such that $Y^k = U\diag(\sigma(Y^k))V^\top$. We denote $H_s$ for a hard-thresholding operator that keeps $s$ of the largest elements (in magnitude) and zeros out the remaining elements.}
    \begin{tabular}{@{}cccc@{}}
    \toprule
          $C$ & Backward step & Comment \\ 
        \midrule
         $\{x \in \bR^n \mid |x_1| = \cdots = |x_n| = r\}$ & $r\sign(y_i)$ & Elementwise sign \\
         $\{x \in \bR^n \mid \|x\|_2 \leq r\}$ & Section~\ref{sec:aniso_ball} & $\ell_2$-ball \\
         $\{x \in \bR^n \mid \|x\|_\infty \leq r\}$ &  $\min(r, \max(-r, y^k_i))$ & Clipping & \\
         $\{x \in \bR^n \mid \|x\|_\infty = r\}$ & Section \ref{sec:aniso_proj_linfty_sphere} & $\ell_\infty$-sphere \\
         $\{x \in \bR^n \mid \|x\|_0 \leq s\}$ & $H_s(y^k)$ & Top $s$ magnitude \\
        \midrule
        $\{X \in \bR^{m\times n} \mid X^\top X = r^2I_n\}$ & $rU_{:,1:n}V^\top$ &  Stiefel manifold\\
       $\{X\in \bR^{m\times n} \mid \|X\|_F \leq r\}$ & Section~\ref{sec:aniso_ball} & Frobenius ball \\
        $\{X \in \bR^{m\times n} \mid \sigma_{\rm max}(X) \leq r\}$ & $U\Diag(\min(r, \sigma_i(Y^k)))V^\top$ & Singular value clipping \\
        $\{X\in \bR^{m\times n} \mid \sigma_{\rm max} = r\}$ & Section \ref{sec:aniso_proj_linfty_sphere} & Spectral sphere \\
         $\{X \in \bR^{m\times n} \mid \rank(X) \leq s\}$ & $U\Diag(H_s(\sigma(Y^k)))V^\top$ & Top $s$ singular values \\
 \bottomrule
    \end{tabular}
\end{table*}

\section{Reference functions for neural networks} \label{sec:matrix_ext}
\paragraph{Matrix setting}
Inspired by the recent success of optimizers that act on 2D parameters, e.g., Muon \cite{jordan2024muon}, spectral descent \cite{carlson2015preconditioned}, Shampoo \cite{gupta2018shampoo}, and others \cite{li2017preconditionedstochasticgradientdescent}, we extend our results to the matrix setting by considering real orthogonal invariant reference functions. Such functions are in one-to-one correspondence with absolutely symmetric functions acting on the singular values (see \cref{fac:abs_symm}) which allows for extensions of convex-analytic results from the vector setting. A first consequence of this observation is that if some absolutely symmetric function $\phi$ satisfies \Cref{assum:sc}, then so does the real orthogonal invariant reference function $\phi \circ \sigma$. This is formalized in the following.
\begin{lemma} \label[lemma]{lem:sc_mat}
   Suppose $\phi : \bR^{\min(m,n)} \to \exR$ is absolutely symmetric and satisfies \Cref{assum:sc}. Then, $\Phi = \phi \circ \sigma : \bR^{m\times n} \to \exR$ satisfies \Cref{assum:sc}. Further, we have that $\Phi^* = \phi^* \circ \sigma$, with $\Phi^* \in \cC^1(\bR^{m\times n})$ and $\nabla \Phi^*(Y) = U\Diag(\nabla \phi^*(\sigma(Y)))V^\top$, where $U\in O(m)$ and $V\in O(n)$ satisfy $Y = U\Diag(\sigma(Y))V^\top$ (full SVD).
\end{lemma}
In particular, by noting that isotropic and anisotropic reference functions are both absolutely symmetric under \Cref{assum:sc}, we obtain the following matrix reference functions with corresponding gradients of the conjugate. We shall refer to these as \emph{spectral isotropic} and \emph{spectral anisotropic} reference functions respectively.
\begin{corollary} \label[corollary]{cor:spectral_ref}
    Suppose $h : \bR \to \exR$ satisfies \Cref{assum:sc}. Then,
    \begin{enumerate}
        \item $\Phi_{\rm iso} = h \circ \|\cdot\| \circ \sigma$ satisfies \Cref{assum:sc} with $\nabla \Phi_{\rm iso}^*(Y) = {h^*}'(\|Y\|)Y/\|Y\|$.
        \item $\Phi_{\rm aniso} : \bR^{m\times n} \to \exR : X \mapsto \sum_{i=1}^{\min(m,n)} h(\sigma_i(X))$  satisfies \Cref{assum:sc} with \[\nabla \Phi_{\rm aniso}^*(Y) = U\Diag({h^*}'(\sigma_1(Y)), \ldots, {h^*}'(\sigma_{\min(m,n)}(Y)))V^\top\] and $U,V$ are defined as in \Cref{lem:sc_mat}.
    \end{enumerate}
\end{corollary}

\begin{example} \label{ex:muon_fwd}
    Taking $h = -\varepsilon(\ln(1-|\cdot|) + |\cdot|)$, and considering its corresponding spectral anisotropic reference function, the forward step in \eqref{eq:update} takes the form \[Y^k = X^k - \gamma_k U^k \Diag\left(\frac{\sigma_1}{\varepsilon + \sigma_1}, \ldots, \frac{\sigma_q}{\varepsilon + \sigma_q}\right){V^k}^\top\] where $q=\min(m,n)$ and $U^k\Diag(\sigma){V^k}^\top$ is any (full) SVD of the direction $D^k$. Since $\frac{x}{\varepsilon + |x|}$ vanishes at zero, the columns in $U,V$ corresponding to zero singular values can be omitted and, as usual, a reduced SVD can be used instead.
\end{example}
\begin{remark} \label{rem:polar_approx}
    Similar as to how the sign gradient method can be seen as an idealized Adam iteration when $\varepsilon\to0^+$, we believe the forward step in \Cref{ex:muon_fwd} is a more natural algorithm of which the lmo approach in \cite{pethick2025generalizedgradientnormclipping} is the ideal variant. Indeed, when $\varepsilon\to0^+$ the forward step becomes the matrix sign (with the convention that zero singular values are mapped to zero) and the Muon update that is most common in the literature is recovered, i.e., $X^{k+1}=X^k-\gamma_k UV^{\top}$ where $D^k = U \Sigma V^\top$ is a reduced SVD of the direction $D^k$. In fact, since the matrix sign is often computed based on compositions of odd polynomials, it is clear that near $0$, these polynomials are unable to approximate the sign very well, being close to $0$ rather than $1$. This discrepancy can naturally be captured by a nonzero $\varepsilon$ as shown in \Cref{fig:pe_approx}, albeit for a slightly different preconditioner. Note that in practice, the input is first normalized to ensure the polynomial iteration converges. We elaborate on this remark in \cref{app:polar_discussion} and provide a convergence result in \cref{subsec:inexact_main_text}.
\end{remark}

Next, if both $G$ and $\Phi$ are real orthogonal invariant matrix functions with additionally $\Phi$ being convex (which holds, in particular, under \cref{assum:sc}), we show how the backward step in \eqref{eq:update} can be reduced to an anisotropic prox of the singular values, thus reducing the computation to the vector setting.
\begin{theorem} \label{thm:aniso_prox_spectral} Let $G, \Phi:\bR^{m\times n}\to\exR$ be real orthogonal invariant. Suppose that $\Phi$ is convex and lsc. Further, let $\gamma > 0$, $Y \in \mathbb{R}^{m\times n}$, and $U\in O(m), V\in O(n)$ be any real orthogonal matrices such that $Y = U\Diag(\sigma(Y))V^\top$.
    If \[x^\star \in \argmin_{x\in\mathbb{R}^{q}}\; G(\Diag x) + (\gamma \star \Phi)(\Diag(x - \sigma(Y)))\]
    with $q=\min(m,n)$, then
    \[
        U\Diag(x^\star)V^\top \in \argmin_{X\in\mathbb{R}^{m\times n}}\; G(X) + (\gamma \star \Phi)(X - Y).
    \]
\end{theorem}

\begin{example}
    Using \cref{thm:aniso_prox_spectral}, the second half of \cref{tab:back_step_euc} follows directly as the spectral counterpart of the first half. Consequently, this framework encompasses the spectral ball and Stiefel manifold constraints considered in \cite{newhouse2025training}, as well as the spectral sphere constraint studied in \cite{xie2026controlled,miyato2018spectral}, and Frobenius constraints as in \cite{loshchilov2024ngpt,wen2025hyperball}. \Cref{fig:bck_steps_vis} visualizes some of our spectral backward steps. 
\end{example}

\paragraph{Layerwise setting}
In neural networks, $E$ often takes the form $E_1 \times \cdots \times E_N$, where each $E_i$ is either $\bR^{n_i}$ (e.g.\ the batch normalization scalars and biases) or $\bR^{m_i\times n_i}$ (e.g.\ transformer weights and embedding matrices). By then considering a separable reference function $\phi(x_1,\ldots,x_N) = \sum_{i=1}^N \phi_i(x_i)$ and a separable regularizer $g(x_1,\ldots,x_N) = \sum_{i=1}^N g_i(x_i)$ with $x_i\in E_i$ for $i=1,\ldots,N$, it follows by \cite[Theorem 4.12]{beck2017first} that $\phi^*(y_1,\ldots,y_N) = \sum_i^N \phi_i^*(y_i)$ and both the forward and backward step in \eqref{eq:update} become separable updates over each $E_i$. 

\section{Analysis in the stochastic setting} \label{sec:stoch}
As already mentioned, we focus mostly on methods akin to normalized gradient and thus for the remainder of this section we assume that $\dom \phi$ is bounded by some ball of radius $D > 0$, i.e., $D = \sup_{x \in \dom \phi}\|x\| < \infty$. As a stationarity measure, we consider $\gap(x^k) := D_{\phi^*}(\nabla f(x^k), -\widetilde \nabla g(x^k))$ where $\widetilde \nabla g(x^k) \in \partial g(x^k)$ is the subgradient chosen at iterate $k$. Due to the strict convexity of $\phi^*$ and the fact that $\nabla \phi^*$ is odd, $\gap(x^k) = 0$ if and only if $-\nabla f(x^k) = \widetilde \nabla g(x^k)$, thus implying stationarity. 

We now move on to analyzing a stochastic version of the algorithm \eqref{eq:update} where we assume that we have access to a stochastic gradient oracle $\nabla f(\cdot, \xi)$ of $\nabla f$, $\xi$ being a random variable. We assume the following.
\begin{assumption} \label{assum:stoch}
    The stochastic gradient oracle $\nabla f(\cdot, \xi)$ is unbiased, i.e., $\expec_\xi[\nabla f(x,\xi)] = \nabla f(x)$ for all $x \in E$. Moreover, it has finite $p$-th central moment,
    \begin{equation}
            \expec_\xi[\|\nabla f(x,\xi)-\nabla f(x)\|^p] \leq \sigma^p, \qquad \forall x \in E
        \end{equation}
    for some  $\sigma \geq 0$ and $p \in (1,2]$.
\end{assumption}
The finite $p$-th central moment assumption is a generalization of the standard bounded variance assumption (which is subsumed for $p=2$) that is standard in analyzing stochastic gradient algorithms. Since we are mostly interested in the case where $g$ involves the indicator of a compact set, in this section we assume that $f$ is a Lipschitz smooth function, while for $g$ merely prox-boundedness.
\begin{assumption} \label{assum:f_lipschitz} $f$ is $L$-Lipschitz smooth;
  $g$ is proper, lsc and prox-bounded with threshold $\gamma_g > \tfrac{\gamma}{\mu}$.
\end{assumption}
Note that we have not yet defined the direction $d^k$, which will take different forms, i.e., standard Polyak momentum and the variance-reduction modification introduced in \cite{cutkosky2019momentum}. Another approach would be to consider using just the estimate $\nabla f(x^k,\xi^k)$ in \eqref{eq:update}. However, due to the nonlinear preconditioning introducing bias in the update \cite{koloskova2023revisiting}, this requires large batch sizes that asymptotically lead to utilizing the true gradient. We therefore do not analyze this case and directly use the momentum estimators that are also more often used in practice.
We now move on to our main convergence result.
\begin{theorem} \label{thm:main_res}
    Let \cref{assum:stoch,assum:f_lipschitz} hold. Let 
    $\{x^k\}_{k \in \bN_0}$ be the sequence of iterates generated by \eqref{eq:update} with $d^0 = \nabla f(x^0, \xi^0)$, 
    \[
        d^k = \alpha \nabla f(x^k, \xi^k) + (1-\alpha)d^{k-1},
    \]
    $\alpha = (K+1)^{-1/2}$ and $\gamma = \bar \gamma(K+1)^{-3/4}$ for some $\bar \gamma > 0$.
    Then, we have
    \begin{equation*}
    \frac{1}{K+1}\sum_{k=0}^K\expec[\gap(x^{k+1})] 
     \leq O\left((K+1)^{-\min\left\{\tfrac{1}{4},\tfrac{p-1}{2p}\right\}}\right)
    \end{equation*}
\end{theorem}
The proof of \cref{thm:main_res} can be found in \cref{app:comp_setting} along with the exact constants for the rates. 

\subsection{Reducing variance with the STORM estimator}
We now describe the convergence of the method using the variance reduction technique from~\cite{cutkosky2019momentum}. In order to properly use the STORM estimator to obtain better guarantees we need to make some more assumptions on our problem data that are nevertheless standard in the related literature \cite{cutkosky2019momentum,yuan2024mars}.
\begin{assumption} \label{assum:fxi_lipschitz} $f(\cdot, \xi)$ is $L$-Lipschitz smooth with probability $1$.
\end{assumption}
We now move on to the main result of this subsection.
\begin{theorem} \label{thm:conv_storm}
    Let \cref{assum:stoch,assum:f_lipschitz,assum:fxi_lipschitz} hold true with $p=2$. Let $\{x^k\}_{k \in \bN_0}$ be the sequence of iterates generated by \eqref{eq:update} with $d^0 = \nabla f(x^0, \xi^0)$,
    \[
        d^k = (1-\alpha_{k})d^{k-1} + \alpha_k \nabla f(x^k, \xi^k) 
        + (1-\alpha_k) (\nabla f(x^k,\xi^k) - \nabla f(x^{k-1},\xi^k)),
    \]
    $\alpha_k = (k+1)^{-2/3}$ and $\gamma_k = (k+1)^{-2/3}$. Then, we have
    \begin{equation*}
        \frac{1}{K+1}\sum_{k=0}^K\expec[\gap(x^{k+1})] 
        \leq \widetilde O\left((K+1)^{-1/3}\right),
    \end{equation*}
    where $\widetilde O$ hides logarithmic factors.
\end{theorem}
The proof of this theorem can be found in \cref{app:storm_proof}, along with the exact constants. It is clear that we obtain a better convergence rate than \cref{thm:main_res}, albeit with one extra sample per iteration and under bounded variance.

\begin{remark}
    Our proof technique for \cref{thm:conv_storm,thm:main_res} departs from existing analyses in the related literature. On the one hand, although we consider a stochastic extension of the method proposed in \cite{laude2025anisotropic}, our analysis differs greatly: in \cite{laude2025anisotropic} convergence follows from standard \emph{majorization-minimization arguments}, while we utilize \emph{adaptive stepsizes} that do not require knowledge of the smoothness parameter. Moreover, due to the nonlinear preconditioning, $\expec[\nabla \phi^*(\nabla f(x^k,\xi^k))] \neq \nabla \phi^*(\nabla f(x^k))$ and thus there is \emph{bias} in the update, making the arguments in the stochastic case far from straightforward. On the other hand, the handling of general nonconvex terms $g$ leads to a different analysis compared to related spectral preconditioning works such as \cite{pethick2025trainingdeeplearningmodels}: the proximal step in \eqref{eq:update} complicates the proof and requires the introduction of a new suboptimality gap. To that aim, we introduce the novel gap $D_{\phi^*}(\nabla f(x^k), -\widetilde \nabla g(x^k))$ that is tailored to the geometry of the algorithm and build our proof based on the machinery of convex analysis \cite{rockafellar1970convex}. Furthermore, we utilize the bounded step of the algorithm to show its resilience to heavy-tailed noise. One potential drawback of our analysis is that we do not obtain convergence rates for the standard stationarity measures. However, in most cases one can convert the obtained rates to the standard ones, see for example the proof of \cref{thm:pe}.
\end{remark}

\subsection{Analysis of Polar Express in the unconstrained setting} \label{subsec:inexact_main_text}
As discussed in \cref{rem:polar_approx}, the update obtained by the Polar Express can be seen as closely approximating a preconditioner $\nabla \phi^*$ after normalizing the input. In this subsection we show how this can be incorporated into our analysis. For that reason, consider the update
\begin{equation} \label{eq:polar_express_update}
    x^{k+1} = x^k - \gamma \nabla \phi^*(d^k_{\epsilon}),
\end{equation}
with $d^k_{\epsilon} = d^k / (\|d^k\|+\epsilon)$ and $\epsilon > 0$. Then we have the following result.
\begin{theorem} \label{thm:pe}
    Let $E=\bR^n$, $g\equiv 0$, \cref{assum:stoch,assum:f_lipschitz} hold true and $\nabla \phi^*(y)_i = \frac{y_i}{(\varepsilon^\kappa + |y_i|^\kappa)^{1/\kappa}}$ for $\kappa > 2, \varepsilon > 0$. Let, moreover, $\{x^k\}_{k \in \bN}$ be the sequence of iterates generated by \eqref{eq:polar_express_update} with $d^k$ as in \cref{thm:main_res} and $\alpha = (K+1)^{-1/2}$, $\gamma = (K+1)^{-3/4}$, $\epsilon \leq (K+1)^{-\min\left\{\tfrac{1}{4},\tfrac{p-1}{2p}\right\}}$. Then, we have
    \begin{equation*}
        \frac{1}{K+1}\sum_{k=0}^K \expec[\|\nabla f(x^k)\|] \leq O\left((K+1)^{-\min\left \{\tfrac{1}{4},\tfrac{p-1}{2p} \right \}}\right).
    \end{equation*}
\end{theorem}

Note that we only study the case $E = \bR^n$ for presentation purposes, since the matrix (and layerwise) case follows similarly. Our analysis is presented in \cref{app:inexact_update,app:stoch_inexact}.

\paragraph{Experiments} The Stiefel manifold and spectral ball (singular value clipping) constraints, together with the algorithm we have analyzed, have been studied empirically in \cite{newhouse2025training} for training Lipschitz transformers and robust MLPs, where it was shown to be very effective also for training stability. In light of \Cref{rem:scion_conn}, our theory also encompasses the state-of-the-art Scion algorithm \cite{pethick2025trainingdeeplearningmodels}, thereby highlighting the unifying power of our framework. Even more recently, the hyperball algorithm \cite{wen2025hyperball} utilizes a Frobenius sphere constraint, though it is not clear whether stationarity holds, cf.\ \cref{rem:stationarity}. Further experiments are provided in \cref{app:experiments} to showcase how weights constraints affect training.

\section*{Acknowledgments}
This work was supported by the Research Foundation Flanders (FWO) PhD grant 11A8T26N and research project G033822N; Research Council KUL grant C14/24/103. This work was supported by the French National Research Agency (ANR) under grant ANR-25-CE23-3749 (project SIMPLES). This work was funded by the Swiss National Science Foundation (SNSF) under grant number 20001-240094.

The resources and services used in this work were provided by the VSC (Flemish Supercomputer Center), funded by the Research Foundation - Flanders (FWO) and the Flemish Government.

\printbibliography


\appendix
\onecolumn
\counterwithin{theorem}{subsection}
\counterwithin{lemma}{subsection}
\counterwithin{definition}{subsection}
\counterwithin{assumption}{subsection}
\counterwithin{proposition}{subsection}

\section{Preliminaries and helper results}
Lipschitz smoothness of a function $f:E \to \bR$ with constant $L > 0$ means that $f$ is continuously differentiable and satisfies
\begin{equation*}
    \|\nabla f(x)-\nabla f(\bar x)\| \leq L \|x-\bar x\| \qquad \forall x, \bar x \in E.
\end{equation*}
Lipschitz smoothness implies the so-called Euclidean descent lemma \cite[Lemma 5.7]{beck2017first}, that describes global quadratic upper bounds
\begin{equation}
    f(x) \leq f(\bar x) + \langle \nabla f(\bar x),x-\bar x \rangle + \tfrac{L}{2}\|x-\bar x\|^2 \qquad \forall x,\bar x \in E.
\end{equation}
\paragraph{Variational and convex analysis results}
Following \cite[Definition 1.77]{Mordukhovich2006}, for a function $g:E \to \exR$ and a point $\bar x \in \dom g$ we denote $\partial g(\bar x)$ its limiting subdifferential at $\bar x$. In the case where $g$ is strictly differentiable at $\bar x$, the subdifferential agrees with its gradient, $\partial g(\bar x) = \{\nabla g(\bar x)\}$, while it agrees with the subdifferential of convex analysis when $g$ is proper, lsc and convex.

Throughout the appendix we also use the following which is an immediate result of \cref{assum:sc}.
\begin{lemma} \label{thm:conv_conj}
    Let $\phi$ satisfy \cref{assum:sc}. Then $\nabla \phi(0) = 0$, $\phi^*\geq 0$, $\phi^*(0) = 0$ and $\phi^*$ is even. Moreover, $\dom \phi^* = E$ and $\phi^*$ is Lipschitz smooth on $E$ with constant $\tfrac{1}{\mu}$. Furthermore, $\phi^*$ is strictly convex and thus $\nabla \phi^*$ strictly monotone. Finally, $\nabla \phi^*$ is odd.
\end{lemma}
\begin{proof}
    $\nabla \phi(0) = 0$ follows from the fact that $\argmin \phi = \{0\}$. $\phi^* \geq 0$ and $\phi^*(0) = 0$ follow from \cite[Proposition 13.22]{bauschke2017correction} and $\phi^*$ is even from \cite[Proposition 13.21]{bauschke2017correction}. $\dom \phi^* = E$ follows from \cite[Proposition 14.15]{bauschke2017correction} and the Lipschitz smoothness of $\phi^*$ on $E$ by \cite[Theorem 18.15]{bauschke2017correction}. The fact that $\phi^*$ is strictly convex follows from \cite[Proposition 18.10]{bauschke2017correction} and the strict monotonicity of $\nabla \phi^*$ from \cite[Proposition 17.10]{bauschke2017correction}. The fact that $\nabla \phi^*$ is odd is obtained by differentiating the function $\phi^* \circ -\id = \phi^*$.
\end{proof}

\paragraph{Matrix analysis and majorization results}
Following result describes a one-to-one correspondence between real orthogonal invariant and absolutely symmetric functions.
\begin{lemma}[{\cite[Proposition 2.2]{lewis1995convex}}] \label{fac:abs_symm} 
     If $F : \bR^{m\times n} \to \exR$ is real orthogonal invariant, then the function $f:\bR^{\min(m,n)}\to\exR$ defined by $f(x) = F(\Diag x)$ is absolutely symmetric. Moreover, we have that $F = f \circ \sigma$.
\end{lemma}
We now recall some notions from the majorization literature. Let $x,y\in\bR^n$ and denote $x^\downarrow$ for the vector obtained by rearranging the entries of $x$ in nonincreasing order, i.e., $x_1^\downarrow \geq \cdots \geq x_n^\downarrow$. Recall that $x$ is weakly majorized by $y$, in symbols $x\prec_w y$, if 
\[
    \sum_{i=1}^k x_i^\downarrow \leq \sum_{i=1}^k y_i^\downarrow, \quad 1\leq k\leq n.
\]
\begin{lemma}[{\cite[Exercise II.2.10]{bhatia2013matrix}}] \label{lem:weak_major_nonneg}
    Let $x, y \in \bR^n$. If $x\geq 0$, $y\geq0$, then
    \[
        x \prec_w y \Longleftrightarrow x\in \conv C_y
    \]
    where $\conv C_y$ denotes the convex hull of $C_y = \{Qy \in \bR^n \mid Q \in \Lambda_n\}$. Here, $\Lambda_n$ denotes the set of all $2^nn!$ generalized permutation matrices, which have exactly one nonzero entry in each row and each column, being either $+1$ or $-1$.
\end{lemma}
Following result describes a generalization of \cite[Problem III.6.13]{bhatia2013matrix} to rectangular matrices and absolutely symmetric, convex, and extended real-valued functions.

\begin{lemma} \label{lem:mirsky_general}
    Let $X,Y\in\bR^{m\times n}$ and denote $q=\min(m,n)$. If $f : \bR^q \to \exR$ is absolutely symmetric and convex, then
    \[
        f(\sigma(X) - \sigma(Y)) \leq f(\sigma(X - Y)).
    \]
\end{lemma}\begin{proof}
    Note that $|\sigma(X) - \sigma(Y)| = Q(\sigma(X) - \sigma(Y))$, with $Q\in\Lambda_{q}$ a diagonal matrix such that each diagonal entry is the sign of the corresponding entry in $\sigma(X) - \sigma(Y)$ (and using the convention $\sign(0)=1$). Since $f$ is absolutely symmetric, we have that $f(Qx) = f(x)$ for all $x\in\bR^q$ and all $Q\in\Lambda_q$, see \cite[Eq.\,(2.5)]{lewis1995convex}, and hence $f(\sigma(X) - \sigma(Y)) = f(|\sigma(X) - \sigma(Y)|)$ .

    Further, it follows from \cite[Theorem 3.4.5]{horn1994topics} that
    \[
        |\sigma(X) - \sigma(Y)| \prec_w \sigma(X-Y).
    \]
    Since $ |\sigma(X) - \sigma(Y)| \geq 0, \sigma(X-Y)\geq0$, \cref{lem:weak_major_nonneg} ensures $|\sigma(X) - \sigma(Y)| \in \conv C_{\sigma(X-Y)}$, and as such $|\sigma(X) - \sigma(Y)| = \sum_{i=1}^N \lambda_i Q_i \sigma(X-Y)$ with $Q_i\in\Lambda_{q}$, $\sum_{i=1}^N \lambda_i = 1$ and $0\leq\lambda_i\leq 1$ for each $i=1,\ldots,N$ where $N = \#C_{\sigma(X-Y)}$.
    
    Lastly, using convexity, and absolute symmetry of $f$ again, we obtain
    \[
        f(\sigma(X) - \sigma(Y)) = f\left(\sum_{i=1}^N \lambda_i Q_i \sigma(X-Y)\right) \\
        \leq \sum_{i=1}^N \lambda_i f(Q_i\sigma(X-Y)) = f(\sigma(X - Y)).
    \]
\end{proof}

\subsection{Computing the anisotropic projection for the \texorpdfstring{$\ell_2$}{l2}-ball constraint} \label{sec:aniso_ball}
We aim to compute
\[
    x^{k+1} \in \argmin_{x \in \bR^n}\; \delta_{\|\cdot\|_2\leq r}(x) + (\gamma_k \star \phi)(x - y^k)
\]
where $\phi(x) = \sum_{i=1}^n h(x_i)$ and $h$ satisfies \Cref{assum:sc}.

If $\|x^k\|_2 \leq r$, then $x^k - y^k = \gamma_k\nabla\phi^*(d^k) \in \intr \dom (\gamma_k \star \phi)$, so the problem has nonempty domain, and since it is moreover strongly convex, it admits a unique minimizer.

Clearly, if $\|y^k\|_2 \leq r$, then $x^{k+1} = y^k$ since $\phi \geq 0$ and $\phi(0) = 0$. If $\|y^k\|_2 > r$, then the KKT conditions require
\[
    h'(\tfrac{1}{\gamma_k}(x_i - y_i^k)) = -2\lambda x_i, \quad i=1,\ldots n \quad \textnormal{and} \quad  \sum_{i=1}^n x_i^2 = r^2 \\
\]
for some $\lambda \geq0$. The first equation is equivalent to $x_i + \gamma_k (h^*)'(2\lambda x_i) = y_i$, which has a unique solution by \cite[Proposition 23.8(iii)]{bauschke2017correction} and the maximal monotonicity of $(h^*)'(2\lambda x)$, see \cite[Theorem 20.25]{bauschke2017correction}. Denoting this solution by $x_i(\lambda)$ and putting this back in the constraint, we arrive at the nonlinear equation
\[
    \sum_{i=1}^n x_i(\lambda)^2 = r^2
\]
in the variable $\lambda$. Further noticing that $|x_i(\lambda)|$ is a decreasing function of $\lambda$ allows us to solve this one-dimensional equation through bisection.

As an example, when $h$ is chosen as in \Cref{example:norm_grad}, then $x_i(\lambda)$ requires solving $x_i + \gamma_k \tfrac{2\lambda x_i}{\varepsilon + |2\lambda x_i|} = y_i^k$ with $\varepsilon > 0$. It is easy to see that $x_i$ will have the same sign as $y_i^k$, so we distinguish between two cases, both of which reduce to finding the root of a quadratic polynomial where the choice of root follows by matching the sign:
\[
    x_i = \sign(y_i^k)\frac{-(\varepsilon + 2\gamma_k\lambda - 2|y_i^k|\lambda) + \sqrt{(\varepsilon + 2\lambda \gamma_k - 2 |y_i^k|\lambda )^2 + 8\lambda\varepsilon |y_i^k|}}{4\lambda}.
\]
The extension to the Frobenius ball constraint for spectral anisotropic reference functions is a straightforward application of \Cref{thm:aniso_prox_spectral}.

\subsection{Computing the anisotropic projection for the \texorpdfstring{$\ell_\infty$}{linfty}-sphere constraint} \label{sec:aniso_proj_linfty_sphere}
We aim to compute
\[
    x^{k+1} \in \argmin_{x \in \bR^n}\; \delta_{\|\cdot\|_\infty=r}(x) + (\gamma_k \star \phi)(x - y^k)
\]
where $\phi(x) = \sum_{i=1}^n h(x_i)$ and $h$ satisfies \Cref{assum:sc}.

We distinguish between two cases. If $\|y^k\|_{\infty} \geq r$, then the solution is clearly given by clipping each entry to the interval $[-r, r]$. Otherwise, the projection consists of finding any element of $y^k$ with maximum magnitude in absolute value, say with index $j$, and changing it to $\sign(y_j^k)r$.

The extension to the spectral sphere constraint for spectral anisotropic reference functions is a straightforward application of \Cref{thm:aniso_prox_spectral}. This process is visualized in \cref{fig:bck_steps_vis}.

\begin{figure}[!t]
    \centering
    \includegraphics[width=0.6\linewidth]{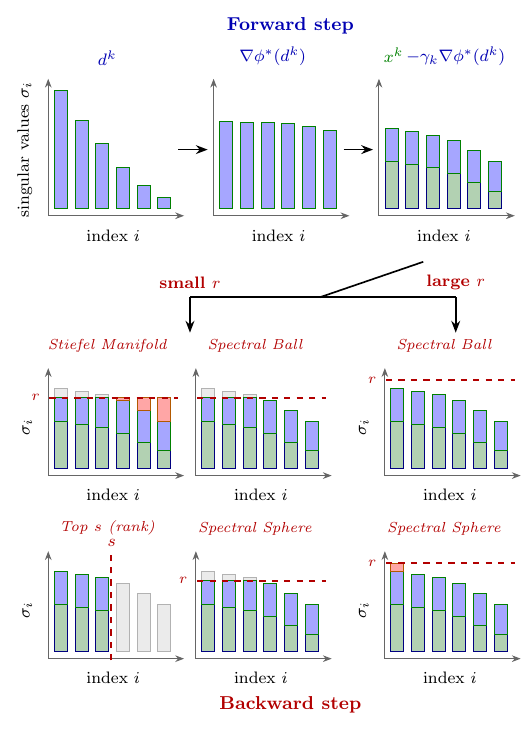}
    \caption{Demonstration of proximal spectral preconditioning with several different \textcolor{\colbck}{backward} steps, summarized also in \cref{tab:back_step_euc}.}
    \label{fig:bck_steps_vis}
\end{figure}

\section{Deterministic analysis} \label{app:det}
In this section we analyze the convergence guarantees of the scheme described in \eqref{eq:update}, extending the analysis of \cite{laude2025anisotropic} to our setting where $\phi$ possibly does not have full domain. Since we are in the deterministic setting, we have $d^k = \nabla f(x^k)$, and for fixed stepsizes $\gamma_k = \gamma$, the update becomes
\begin{equation*}
    \begin{aligned}
        y^k &= x^k - \gamma \nabla \phi^*(\nabla f(x^k))
        \\
        x^{k+1} & \in \argmin_{x \in E}\; g(x) + (\gamma \star \phi)(x-y^k),
    \end{aligned}
\end{equation*}
starting from $x^0 \in \dom g$. In this setting, the main assumption on $f$ is anisotropic smoothness.
\begin{definition}[Anisotropic smoothness]
    Let $f \in \cC^1(E)$. We say that $f$ is anisotropically smooth relative to $\phi$ with constant $L > 0$ if for all $x, \bar x \in E$
    \begin{equation} \label{eq:extended_ad} 
        \begin{aligned}
            f(x) &\leq f(\bar x) + (\tfrac{1}{L} \star\phi)(x - \bar y) - (\tfrac{1}{L} \star \phi) (\bar x - \bar y),
        \end{aligned}
    \end{equation}
    where $\bar y = \bar x - L^{-1} \nabla \phi^*(\nabla f(\bar x))$.
\end{definition}
\begin{assumption}\label{assum:f}
    $f$ is anisotropically smooth relative to $\phi$ with constant $L > 0$.
\end{assumption}
Moving on to $g$, we will work under the standard setting of \emph{prox-boundedness} from \cite[Definition 1.23]{RoWe98}:
\begin{definition}
    A function $g:E \to \exR$ is prox-bounded if there exists $\gamma > 0$ such that $\inf_{x \in E} g(x) + \tfrac{1}{2\gamma}\|x-\bar x\|^2 > -\infty$ for some $\bar x \in E$. The supremum of the set of all such $\gamma$ is the threshold $\gamma_g$ of prox-boundedness of $g$.
\end{definition}
We will utilize prox-boundedness in order to prove the continuity of the \emph{anisotropic Moreau envelope}, defined as
\begin{equation}
    g^\phi_\gamma(y) := \inf_{x \in E} g(x) + (\gamma \star \phi)(x-y) = \inf_{x\in E}g(x) + (\gamma \star \phi)(y-x),
\end{equation}
where the equality follows by the fact that $\phi$ is even. It is clear thus that $g^\phi_\gamma$ is an \emph{infimal convolution} \cite[Definition 12.1]{bauschke2017correction}, also known as \emph{epi-addition} \cite[p.\,23]{RoWe98}. In light of \cite[Proposition 12.6(ii)]{bauschke2017correction}, $\dom g^\phi_\gamma = \dom g + \gamma \dom \phi$. We next show the continuity of $g^\phi_\gamma$ relative to its domain. Our analysis is a specialization of the analysis in \cite{laude2019optimization,laude2021lower}:
\begin{proposition} \label{thm:prox_bound}
    Let $g$ be proper, lsc and prox-bounded with threshold $\gamma_g$. Then, for any $\gamma \in (0,\mu \gamma_g)$, the function $g^\phi_\gamma$ is continuous relative to $\dom g^\phi_\gamma = \dom g + \gamma \dom \phi$.
\end{proposition}
\begin{proof}
    First, we are going to show that the function $g(x) + (\gamma \star \phi)(y-x)$ is level-bounded in $x$ locally uniformly in $y$ \cite[Definition 1.16]{RoWe98}. To that aim, assume the contrary: let $\gamma \in (0, \mu \gamma_g)$ and assume that there exist $\bar y \in E$, $\alpha \in \bR$ and sequences $y^\nu \to \bar y$ and an unbounded $\{x^\nu\}_{\nu \in \bN}$ such that for $\nu$ sufficiently large
    \begin{equation} \label{eq:contr}
        g(x^\nu) + \gamma \phi(\tfrac{1}{\gamma}(y^\nu - x^\nu)) \leq \alpha.
    \end{equation}
    Let $\gamma' \in (\gamma, \mu \gamma_g)$ and set $\gamma_0 = \tfrac{\gamma'}{\mu}$. Since $\gamma_0 < \gamma_g$, similarly to the proof of \cite[Theorem 1.25]{RoWe98}, there exist $\gamma_1 > \gamma_0$ and $\beta \in \bR$ such that $g(\cdot) + \tfrac{1}{2\gamma_1}\|\cdot\|^2 \geq \beta$. Let $\nu\in\bN$. It follows that
    \begin{align*}
        g(x^\nu) + \tfrac{1}{2\gamma_0}\|x^\nu - y^\nu\|^2 &= g(x^\nu) +\tfrac{1}{2\gamma_1}\|x^\nu\|^2 + \left(\tfrac{1}{2\gamma_0} - \tfrac{1}{2\gamma_1}\right)\|x^\nu\|^2 + \tfrac{1}{2\gamma_0}\|y^\nu\|^2 - \tfrac{1}{\gamma_0}\langle x^\nu, y^\nu\rangle \\
        &\geq \beta + \left(\tfrac{1}{2\gamma_0} - \tfrac{\gamma_1}{2(\gamma_1 - \gamma_0)\gamma_0} \right)\|y^\nu\|^2
    \end{align*}
    where the equality is simple algebraic manipulation and the inequality follows from Young's inequality $\langle x^\nu, y^\nu\rangle \leq \frac{\lambda^2}{2}\|x^\nu\|^2 + \frac{1}{2\lambda^2}\|y^\nu\|^2$ with $\lambda^2 = 1 - \frac{\gamma_0}{\gamma_1} > 0$. Since $\{y^\nu\}_{\nu\in\bN}$ is bounded, it follows that there exists some $\bar \beta \in \bR$ such that for all $\nu \in \bN:$
    \[
        g(x^\nu) + \tfrac{1}{2\gamma_0}\|x^\nu - y^\nu\|^2 \geq \bar \beta.
    \]
    By the strong convexity gradient inequality for $\phi$ at points $\tfrac{1}{\gamma'}(y^\nu-x^\nu)$ and $0 \in \intr \dom \phi$, and using that $\nabla \phi(0) = 0$ and $\phi(0)=0$:
    \begin{equation*}
        \gamma'\phi(\tfrac{1}{\gamma'}(y^\nu-x^\nu)) \geq \gamma' \tfrac{\mu}{2}\|\tfrac{1}{\gamma'}(y^\nu-x^\nu)\|^2 = \tfrac{1}{2 \gamma_0}\|y^\nu - x^\nu\|^2.
    \end{equation*}
    Therefore,
    \begin{equation*}
        g(x^\nu) + \gamma'\phi(\tfrac{1}{\gamma'}(y^\nu-x^\nu)) \geq \bar \beta.
    \end{equation*}
    Multiplying with $-1$ and adding to \eqref{eq:contr} yields, for $\nu$ large enough:
    \[
        \gamma \phi(\tfrac{1}{\gamma}(y^\nu - x^\nu)) - \gamma'\phi(\tfrac{1}{\gamma'}(y^\nu-x^\nu)) \leq \alpha - \bar \beta  
    \]
    and from the strong convexity inequality for $\phi$ between the points $0$ and $\tfrac{1}{\gamma}(y^\nu - x^\nu)$,
    \[
        \gamma' \phi(\tfrac{1}{\gamma'}(y^\nu-x^\nu)) = \gamma' \phi(\tfrac{\gamma}{\gamma'}\tfrac{1}{\gamma}(y^\nu-x^\nu)) \leq \gamma \phi(\tfrac{1}{\gamma}(y^\nu-x^\nu)) - \tfrac{\mu}{2}\gamma' \tfrac{\gamma}{\gamma'}(1-\tfrac{\gamma}{\gamma'})\|\tfrac{1}{\gamma}(y^\nu - x^\nu)\|^2.
    \]
    Substituting now this result in the previous inequality and taking $\nu \to \infty$ we are led to $\infty \leq \alpha - \bar \beta$, a contradiction. Therefore, $g(x) + (\gamma \star \phi)(y-x)$ is level-bounded in $x$ locally uniformly in $y$ and since it is also proper and lsc, the claimed result follows from \cite[Theorem 1.17]{RoWe98}.
\end{proof}
We can now move on to formulating the assumption on $g$.
\begin{assumption} \label{assum:g} 
    $g$ is proper, lsc and prox-bounded with threshold $\gamma_g > \tfrac{\gamma}{\mu}$.
\end{assumption}
In our analysis, we are going to utilize the regularized gap function from \cite{laude2025anisotropic}, that takes the following form:
\begin{equation} \label{eq:gap}
    \cG^\phi_\gamma(x) = \tfrac{1}{\gamma}\left [ g(x) + \gamma \phi(\nabla \phi^*(\nabla f(x))) - g^\phi_\gamma(x- \gamma \nabla \phi^*(\nabla f(x)))\right ].
\end{equation}
In the following proposition we describe the properties of $\cG^\phi_\gamma$ that make it a good stationarity measure, generalizing \cite[Lemma 5.1]{laude2025anisotropic}.
\begin{proposition} \label{thm:gap_func}
    For any $\gamma < \mu \gamma_g$, the gap function $\cG^\phi_\gamma$ is nonnegative and lsc relative to $\dom g$. Moreover, any $x^\star \in \dom g$ such that $\cG^\phi_\gamma(x^\star) = 0$ is a stationary point of $F$.
\end{proposition}
\begin{proof}
    To begin with, note that for any $x \in \dom g$, $\cG^\phi_\gamma(x) \in \bR$ since $x - \gamma \nabla \phi^*(\nabla f(x)) \in \dom g + \gamma \dom \phi = \dom g^\phi_\gamma$ from \cref{thm:prox_bound}. Moreover,
    \[
        g^\phi_\gamma(x - \gamma \nabla \phi^*(\nabla f(x))) = \inf_{z\in E}g(z) + (\gamma \star \phi)(z-(x - \gamma \nabla \phi^*(\nabla f(x)))) \leq g(x) + \gamma \phi(\nabla \phi^*(\nabla f(x))),
    \]
    where the inequality follows by setting $z = x$ in the infimum. Therefore, $\cG^\phi_\gamma$ is nonnegative on $\dom g$. Now, $g$ is proper and lsc, while $\nabla \phi^* \circ \nabla f$ maps $E$ to $\intr \dom \phi$ where $\phi$ is $\cC^1$ and thus continuous, implying that $\phi \circ \nabla \phi^* \circ \nabla f$ is continuous on $\dom g$. Moreover, $\gamma < \mu \gamma_g$ and thus from \cref{thm:prox_bound}, $g^\phi_\gamma \circ (\id - \gamma \nabla \phi^* \circ \nabla f)$ is continuous relative to $\dom g$. Therefore $\cG^\phi_\gamma$ is lsc relative to $\dom g$ being the sum of lsc and proper functions.

    For the final claim, consider a point $x^\star \in \dom g$ such that $\cG^\phi_\gamma(x^\star) = 0$. Then, we have
    \begin{equation*}
        g(x^\star) + \gamma \phi(\nabla \phi^*(\nabla f(x^\star))) = \inf_{z\in \bR^n}g(z) + (\gamma \star \phi)(z-x^\star+\gamma \nabla \phi^*(\nabla f(x^\star))),
    \end{equation*}
    i.e., the infimum is achieved at $x^\star$. Therefore, from the modern version of Fermat's theorem \cite[Proposition 1.114]{Mordukhovich2006} we have that
    \begin{equation*}
        0 \in \partial (g + (\gamma \star \phi)(\cdot-x^\star+\gamma \nabla \phi^*(\nabla f(x^\star))))(x^\star).
    \end{equation*}
    Now note that $\tfrac{1}{\gamma}(x^\star - x^\star+ \gamma \nabla \phi^*(\nabla f(x^\star))) = \nabla \phi^*(\nabla f(x^\star)) \in \intr \dom \phi$. From the definition of the interior of a set we have the existence of a neighborhood $B \subset \intr \dom \phi$ of $\nabla \phi^*(\nabla f(x^\star))$ such that $\phi$ is $\cC^1(B)$ by \cref{assum:sc}. Therefore, in light of \cite[p.\,19]{Mordukhovich2006} it is strictly differentiable at $x^\star$. Since, moreover, $g$ is finite at $x^\star$ we obtain from \cite[Proposition 1.107]{Mordukhovich2006}:
    \begin{equation*}
        0 \in \partial g(x^\star) + \nabla f(x^\star).
    \end{equation*}
    This then implies $0 \in \partial F(x^\star)$ which concludes our proof.
\end{proof}
We can now move on to the convergence result.
\begin{theorem}\label{thm:conv_result} Suppose that \Cref{assum:f,assum:g} hold. 
    Let $\{x^k\}_{k \in \bN_0}$ be the sequence of iterates generated by \eqref{eq:update} with $d^k = \nabla f(x^k)$ and $\gamma \leq \tfrac{1}{L}$. Then, the following sufficient decrease property holds true for all $k \in \bN$:
    \begin{equation} \label{eq:suff_decr}
        F(x^{k+1}) \leq F(x^k) - \gamma\cG^\phi_\gamma(x^k).
    \end{equation}
    Moreover, every limit point of $\{x^k\}_{k \in \bN}$ is a stationary point of $F$ and the minimum over the past gaps vanishes sublinearly:
    \begin{equation}
        \min_{k \in \{0,1,\dots,K-1\}}\cG^\phi_\gamma(x^k) \leq \frac{1}{\gamma K}(F(x^0)- F_\star)
    \end{equation}
\end{theorem}
\begin{proof}
    Since $\frac{1}{\gamma} \geq L$, we have by the monotonicity property of anisotropic smoothness under episcaling \cite[Proposition B.3]{oikonomidis2025nonlinearly} that 
    \begin{align*}
        F(x^{k+1}) 
        &\leq f(x^k) + (\gamma \star \phi)(x^{k+1} - x^k + \gamma \nabla \phi^*(\nabla f(x^k))) - \gamma \phi(\nabla \phi^*(\nabla f(x^k))) + g(x^{k+1})
        \\
        &=F(x^k) - (g(x^k) + \gamma \phi(\nabla \phi^*(\nabla f(x^k))) \\
        &\mspace{220mu}- (\gamma \star \phi)(x^{k+1} - x^k + \gamma \nabla \phi^*(\nabla f(x^k))) - g(x^{k+1}))
        \\
        &= F(x^k) - \gamma\cG^\phi_\gamma(x^k),
    \end{align*}
    where the first equality follows by adding and subtracting $g(x^k)$ and the second by the update of $x^{k+1}$. 

    Summing up \eqref{eq:suff_decr} until $K-1$ we get
    \begin{equation*}
        \sum_{k=0}^{K-1}\cG^\phi_\gamma(x^k) \leq \tfrac{1}{\gamma}(F(x^0)-F_\star).
    \end{equation*}
    The bound on the stationarity measure follows by the inequality \[\sum_{k=0}^{K-1}\cG^\phi_\gamma(x^k) \geq K \min_{k\in \{0,1,\dots, K-1\}} \cG_\gamma^\phi (x^k).\]
    This implies that
    \begin{equation}
        \lim_{k \to \infty}\cG^\phi_\gamma(x^k) = 0.
    \end{equation}
    Now let $x^\star$ be a limit point of $\{x^k\}_{k \in \bN_0}$ and $x^{k_j} \to x^\star$ be a corresponding subsequence. Since $\cG_\gamma^\phi$ is lsc relative to $\dom g$ and $x^k,x^\star \in \dom g$ we have
    \begin{equation*}
        0\leq \cG^\phi_\gamma(x^\star) \leq \lim_{j \to \infty}\cG^\phi_\gamma(x^k) = 0.
    \end{equation*}
    Therefore, $\cG^\phi_\gamma(x^\star) = 0$ and from \cref{thm:gap_func} we have that $x^\star$ is a stationary point of $F$.
\end{proof}

If $\dom \phi$ is bounded, then we have the following properties regarding the iterates of \eqref{eq:update}.
\begin{proposition} \label{thm:prox_thm}
    Let $x^{k+1}$ be generated by \eqref{eq:update} with $d^k \in E$ and suppose $\sup_{x \in \dom \phi}\|x\| = D < +\infty$. Then, there exists $\widetilde \nabla g(x^{k+1}) \in \partial g(x^{k+1})$ such that 
    \begin{equation} \label{eq:subg_char}
        x^{k+1}-x^k = -\gamma_k \nabla \phi^*(d^k) + \gamma_k \nabla \phi^*(-\widetilde \nabla g(x^{k+1})) = -\gamma_k \nabla \phi^*(d^k) - \gamma_k \nabla \phi^*(\widetilde \nabla g(x^{k+1})).
    \end{equation}
    This further implies that $\|x^{k+1}-x^k\| \leq 2\gamma_k D$. Moreover, 
    \begin{equation}
        \langle d^k,x^{k+1}-x^k \rangle \leq -\tfrac{\mu}{2\gamma_k}\|x^{k+1}-x^k\|^2 + g(x^k) - g(x^{k+1}).
    \end{equation}
\end{proposition}
\begin{proof}
    To begin with, note that since $y^k \in \dom g + \gamma \intr(\dom \phi) \subseteq \dom g^\phi_\gamma$ the minimization problem is well-defined in light of \cref{thm:prox_bound}.
    From the optimality conditions of the minimization problem we have that
    \begin{align*}
        -\nabla \phi(\tfrac{1}{\gamma_k}(x^{k+1}-y^k)) \in \partial g(x^{k+1}).
    \end{align*}
    Denoting $\widetilde \nabla g(x^{k+1}) \in \partial g(x^{k+1})$ the subgradient for which the above inclusion holds, we have
    \begin{align*}
        x^{k+1}-x^k = -\gamma_k \nabla \phi^*(d^k) + \gamma_k \nabla \phi^*(-\widetilde \nabla g(x^{k+1})),
    \end{align*}
    while the second equality in \eqref{eq:subg_char} follows by the fact that $\nabla \phi^*$ is odd. Since $\ran \nabla \phi^* \subseteq \intr \dom \phi$, the claimed bound on $\|x^{k+1}-x^k\|$ follows. Further, since $x^{k+1} \in \argmin_{x \in E} g(x) + (\gamma_k \star \phi)(x-y^k)$ we obtain:
    \begin{equation} \label{eq:min_func}
        g(x^{k+1}) +( \gamma_k \star \phi)(x^{k+1}-y^k) \leq g(x^k) + (\gamma_k \star \phi)(x^k-y^k).
    \end{equation}
    From the strong convexity gradient inequality for $\gamma_k \phi$ between points $\tfrac{1}{\gamma_k}(x^{k+1}-y^k)$ and $\tfrac{1}{\gamma_k}(x^k-y^k)$ we further have:
    \begin{multline*}
        (\gamma_k \star \phi)(x^{k+1}-y^k) \geq (\gamma_k \star \phi)(x^k-y^k) + \gamma_k \langle \nabla \phi(\tfrac{1}{\gamma_k}(x^k-y^k)),\tfrac{1}{\gamma_k}(x^{k+1}-x^k) \rangle \\ + \tfrac{\gamma_k \mu}{2}\|\tfrac{1}{\gamma_k}(x^{k+1}-x^k)\|^2.
    \end{multline*}
    Moreover, $\tfrac{1}{\gamma_k}(x^k-y^k) = \nabla \phi^*(d^k)$ and thus
    \begin{equation*}
        (\gamma_k \star \phi)(x^{k+1}-y^k) \geq (\gamma_k \star \phi)(x^k-y^k) + \langle d^k,x^{k+1}-x^k \rangle + \tfrac{\mu}{2\gamma_k}\|x^{k+1}-x^k\|^2.
    \end{equation*}
    Using this inequality in \eqref{eq:min_func} and rearranging, we obtain the claimed result.
\end{proof} 
Note that the result above holds regardless of the choice of the direction $d^k$ and thus applies to the stochastic setting as well.
\subsection{Failure of the proximal step under nonsmoothness of \texorpdfstring{$\phi^*$}{phi*}}
We remark that in the unconstrained deterministic case, if we additionally assume that the objective is Lipschitz smooth and $\sup_{x \in \dom \phi}\|x\| = D < +\infty$,  the convergence of the method can be proven using techniques similar to those in the literature e.g.\ \cite{pethick2025trainingdeeplearningmodels}. Indeed, note that from the Euclidean descent lemma and for $\gamma_k=\gamma$ constant we have
\begin{align*}
    f(x^{k+1}) 
    &\leq f(x^k) + \langle \nabla f(x^k),x^{k+1}-x^k \rangle + \tfrac{L}{2}\|x^{k+1}-x^k\|^2
    \\
    &= f(x^k) -\gamma \langle \nabla f(x^k),\nabla \phi^*(\nabla f(x^k)) \rangle + \tfrac{\gamma^2 L }{2}\|\nabla \phi^*(\nabla f(x^k))\|^2
    \\
    & \leq f(x^k) - \gamma [\phi^*(\nabla f(x^k)) + \phi(\nabla \phi^*(\nabla f(x^k)))] + \tfrac{\gamma^2LD^2}{2},
\end{align*}
where the inequality follows by the bound on the diameter of $\dom \phi$ and the equality case of the Fenchel--Young inequality. Using thus a suitable $\gamma$ we can obtain a convergence rate for the corresponding stationarity measure $\phi^*(\nabla f(x^k)) + \phi(\nabla \phi^*(\nabla f(x^k)))$. In this case we could also use reference functions $\phi$ that do not satisfy our smoothness assumptions, such as the indicator of some convex, compact set with the most standard choice being the ball of some norm. Nevertheless, this is not the case under the presence of the additional nonsmooth term $g$. Indeed, suppose that $g$ is convex. Then we have the following inequality.
\[
    f(x^{k+1}) 
    \leq f(x^k) + \langle \nabla f(x^k) + \widetilde \nabla g(x^{k+1}),x^{k+1}-x^k \rangle - \langle \widetilde \nabla g(x^{k+1}),x^{k+1}-x^k \rangle  + \tfrac{L}{2}\|x^{k+1}-x^k\|^2
\]
where $\widetilde \nabla g(x^{k+1}) \in \partial g(x^{k+1})$, and therefore we obtain:
\begin{align*}
    F(x^{k+1}) &\leq F(x^k) + \langle \nabla f(x^k) + \widetilde\nabla g(x^{k+1}), x^{k+1} - x^k\rangle + \tfrac{L}{2}\|x^{k+1} - x^k\|^2 \\
    &\leq F(x^k) - \gamma \langle \nabla f(x^k)+\widetilde \nabla g(x^{k+1}),\nabla \phi^*(\nabla f(x^k))+ \nabla \phi^*(\widetilde \nabla g(x^{k+1}))\rangle + 2L\gamma^2  D^2.
\end{align*}
where the first inequality utilizes the convex subgradient inequality for $g$ and the second inequality follows from \Cref{thm:prox_thm}.
It is clear that unless more properties are assumed on $\phi$, the inner product in the inequality above can be $0$ even if $-\nabla f(x^k) \notin \partial g(x^{k+1})$.

\section{Spectral reference functions and anisotropic projections}

\subsection{Proof of \texorpdfstring{\Cref{lem:sc_mat}}{Lemma 3.1}}
\begin{proof}
    First, we show that $\Phi$ is $\mu$-strongly convex and lsc. Since $\phi$ is assumed to be $\mu$-strongly convex, it follows from \cite[Theorem 5.17]{beck2017first} that $\phi - \tfrac{\mu}{2}\|\cdot\|^2_2$ is convex. Then, since $\Phi - \tfrac{\mu}{2} \|\cdot\|^2_F = \phi \circ \sigma - \tfrac{\mu}{2}\|\cdot\|_2^2 \circ \sigma= (\phi - \tfrac{\mu}{2}\|\cdot\|_2^2) \circ \sigma$, \cite[Corollary 2.6]{lewis1995convex} ensures that $\Phi - \tfrac{\mu}{2} \|\cdot\|^2_F$ is convex, and using the reverse direction of \cite[Theorem 5.17]{beck2017first} yields the desired result. That $\Phi$ is lsc also follows from \cite[Corollary 2.6]{lewis1995convex}.

    Second, since $\Phi = \phi \circ \sigma$ is real orthogonal invariant from \cite[Proposition 2.2]{lewis1995convex}, we have that $\Phi(-X) = \Phi((-I_m) X I_n) = \Phi(X)$ for all $X\in\bR^{m\times n}$, thus $\Phi$ is even. Furthermore, $\Phi(0) = (\phi \circ \sigma)(0) = \phi(0) = 0$ by the assumption on $\phi$. Denote $C = \intr\dom\Phi$. That $C\neq\emptyset$, $\Phi\in\cC^1(C)$ and that $\|\nabla \Phi(X^\nu)\| \to \infty$ whenever $C \ni X^\nu \to X \in \bdry \dom \Phi$ follows from the hypothesis on $\phi$, \cite[Corollary 3.2]{lewis1995convex} and \cite[Theorem 26.1]{rockafellar1970convex}.

    Lastly, the formula for the conjugate follows from \cite[Theorem 2.4]{lewis1995convex}, that $\Phi^* \in \cC^1(\bR^{m\times n})$ is clear from the duality of strong convexity and Lipschitz smoothness, and the form of the gradient is the content of \cite[Theorem 3.1]{lewis1995convex}. 
\end{proof}

\subsection{Proof of \texorpdfstring{\Cref{thm:aniso_prox_spectral}}{Theorem 3.5}}
\begin{proof}
Let $G:\bR^{m\times n}\to\exR$ and $\Phi:\bR^{m\times n}\to\exR$ be real orthogonal invariant functions, and assume moreover that $\Phi$ is convex and lsc. Suppose $U\Diag(\sigma(Y))V^\top= Y \in \bR^{m\times n}$ is any full SVD of $Y$ with $U\in O(m), V\in O(n)$. We aim to find $X\in\bR^{m\times n}$ that minimizes $G(X) +( \gamma \star \Phi)(X - Y)$. First, by real orthogonal invariance, it follows that for any $X\in\bR^{m\times n}$
    \begin{align*}
        G(X) + (\gamma \star \Phi)(X - Y) &=  G(U^\top XV) + (\gamma \star \Phi)(U^\top (X - Y)V) \\
        &= G(U^\top XV) + (\gamma \star \Phi)(U^\top XV - \Diag(\sigma(Y)).
    \end{align*}
    So, by a change of variables $\tilde{X} = U^\top XV$,
    \[
        \argmin_{X\in\bR^{m\times n}} G(X) + (\gamma \star \Phi)(X - Y) = U\argmin_{\tilde{X} \in \bR^{m\times n}} \left\{ G(\tilde{X}) + (\gamma \star \Phi)(\tilde{X} - \Diag(\sigma(Y)))\right\} V^\top.
    \]
    Now, define $g:\bR^q \to \exR$ and $\phi:\bR^q\to\exR$ through $g(x) = G(\Diag x)$ and $\phi(x) = \Phi(\Diag x)$ respectively. Both $g$ and $\phi$ are absolutely symmetric and satisfy $G = g\circ \sigma$ and $\Phi=\phi\circ\sigma$ in light of \Cref{fac:abs_symm}. Moreover, $\phi$ is convex and lsc through \cite[Corollary 2.6]{lewis1995convex}.
    
    Setting $f$ in \cref{lem:mirsky_general} to $\gamma \star \phi$ yields for any $\tilde{X}\in\bR^{m\times n}$ that
    \[
        (g\circ \sigma)(\tilde{X}) + (\gamma \star \phi)(\sigma(\tilde{X}) - \sigma(Y)) \leq (g\circ \sigma)(\tilde{X}) + (\gamma \star \Phi) (\tilde{X} - \diag(\sigma(Y))).
    \]
    Since $\sigma(\tilde{X}) \in \bR^q$, we can further lower bound by
    \[
        (g\circ \sigma)(\tilde{X}) + (\gamma \star \Phi) (\tilde{X} - \diag(\sigma(Y))) 
        \geq \inf_{x\in\bR^q} g(x) + (\gamma \star \phi)(x - \sigma(Y)).
    \]
    Therefore,
    \[
        \inf_{\tilde{X}\in\bR^{m\times n}}  G(\tilde{X}) + (\gamma \star \Phi)(\tilde{X} - \Diag(\sigma(Y))) \geq \inf_{x\in\bR^q} g(x) +( \gamma \star \phi)(x - \sigma(Y)).
    \]
    We now show that if
    \[
        x^\star \in  \argmin_{x\in\bR^q} \{G(\Diag(x)) + (\gamma \star \Phi)(\Diag(x-\sigma(Y))) \} =  \argmin_{x\in\bR^q} \{g(x) + (\gamma \star \phi)(x-\sigma(Y)) \} \neq \emptyset,
    \]
    then $\Diag(x^\star)$ attains this lower bound. Indeed, this follows easily from the fact that $g$ and $\phi$ are absolutely symmetric:
    \begin{align*}
        G(\Diag(x^\star)) &= g(|x^\star|^\downarrow) = g(x^\star), \\ (\gamma \star \Phi)(\Diag(x^\star) - \Diag(\sigma(Y))) &= (\gamma \star \phi )(|x^\star - \sigma(Y)|^\downarrow) = (\gamma \star \phi)(x^\star - \sigma(Y)).
    \end{align*}
    We conclude that
    \[
        U\Diag(x^\star)V^\top \in \argmin_{X\in\bR^{m\times n}} G(X) +( \gamma \star \Phi)(X-Y).
    \]
\end{proof}

\section{Analysis in the stochastic setting}
Throughout this section we assume that $f$ is Lipschitz smooth with constant $L > 0$. We remind that $\sup_{x \in \dom \phi}\|x\| = D < \infty$ and thus $\|\nabla \phi^*(y)\| \leq D$ for all $y \in E$. This moreover implies that $\phi^*$ is a $D$-Lipschitz continuous function on $E$. By a simple application of the Lipschitz continuity inequality, this implies
\begin{equation} \label{eq:rev_ineq}
    \phi^*(y) \geq \phi^*(\bar y) - D\|y-\bar y\| \qquad \forall y, \bar y \in E.
\end{equation}

\subsection{Warm-up: The smooth case}
When using the linear minimization oracle (and thus in our notation $\phi$ is the indicator of some norm ball) the analysis follows naturally from the definition of the dual norm, see e.g. \cite[Lemma D.1]{pethick2025trainingdeeplearningmodels}. This is not the case in our setting where $\phi$ is assumed to satisfy \cref{assum:sc}. However, we can use the definition of the convex conjugate, that in a way generalizes the definition of the dual norm, to obtain similar convergence guarantees. 

For the remainder of this subsection we assume that $\phi^* + \phi \circ \nabla \phi^*$ is $D$-Lipschitz continuous with constant $D$ on $E$, which is in fact the case for all the interesting reference functions we work with. Moreover, $g$ is assumed to be identically zero and for $\gamma_k=\gamma$ constant, the algorithm takes the form
\begin{equation} \label{eq:alg_unc}
    x^{k+1} = x^k - \gamma \nabla \phi^*(d^k).
\end{equation}
From the Euclidean descent lemma for $f$ between points $x^k$ and $x^{k+1}$, we have:
\begin{align*}
    f(x^{k+1}) 
        &\leq f(x^k) -\gamma \langle \nabla f(x^k),\nabla \phi^*(d^k) \rangle + \tfrac{L\gamma^2 D^2}{2}
        \\
        &= f(x^k) + \gamma \langle d^k-\nabla f(x^k),\nabla \phi^*(d^k) \rangle - \gamma \langle d^k,\nabla \phi^*(d^k) \rangle +  \tfrac{L\gamma^2 D^2}{2}
        \\
        & \leq f(x^k) + \gamma D\|\lambda^k\| - \gamma [\phi^*(d^k)+\phi(\nabla \phi^*(d^k))] + \tfrac{L\gamma^2 D^2}{2}
        \\
        & \leq f(x^k) + 2\gamma D\|\lambda^k\| - \gamma [\phi^*(\nabla f(x^k))+\phi(\nabla \phi^*(\nabla f(x^k)))] + \tfrac{L\gamma^2 D^2}{2},
\end{align*}
where in the first inequality we used the update \eqref{eq:alg_unc} along with the bound $\|\nabla \phi^*(d^k)\| \leq D$, and in the second inequality the Cauchy--Schwarz inequality, the equality case of the Fenchel--Young inequality and $\lambda^k \coloneq d^k - \nabla f(x^k)$. The final inequality follows by \eqref{eq:rev_ineq} for the function $\phi^* + \phi \circ \nabla \phi^*$ between points $y = d^k$ and $\bar y = \nabla f(x^k)$. The analysis can now follow any of the ones in the related literature that bound the term $\expec[\|\lambda^k\|]$ but now for the stationarity measure $\phi^*(\nabla f(x^k))+\phi(\nabla \phi^*(\nabla f(x^k)))$. 

\textbf{Norm-constrained lmos}. It is clear that \cref{assum:sc} and in general the smoothness of $\phi^*$ is not required, and we could have $\phi = \delta_B$ with $B$ the unit ball of some norm $\|\cdot\|$ similarly to \cite{pethick2025trainingdeeplearningmodels}. In this case the stationarity gap would just collapse to $\|\nabla f(x^k)\|_*$ where $\|\cdot \|_*$ is now the dual norm, thus leading to standard convergence guarantees.

\subsection{The composite case} \label{app:comp_setting}
We now return to the composite setting to analyze our proposed scheme. We begin with a simple lemma that leverages the Euclidean descent inequality along with convex conjugacy.
\begin{lemma} \label{thm:comp_desc1}
    Let $\{x^k\}_{k \in \bN}$ be the sequence of iterates generated by \eqref{eq:update}. Then, the following inequality holds
    \begin{multline}
        F(x^{k+1}) \leq F(x^k) + 2\gamma_k D\|\lambda^k\| + 2\gamma_k^2 D^2L + \gamma_k\phi^*(\widetilde \nabla g(x^{k+1})) - \gamma_k\phi^*(d^k)
        \\
        -\gamma_k \langle \nabla \phi^*(\widetilde \nabla g(x^{k+1})),\widetilde \nabla g(x^{k+1}) \rangle - \gamma_k \langle d^k,\nabla \phi^*(\widetilde \nabla g(x^{k+1})) \rangle \label{eq:main_ineq_bregman},
    \end{multline}
    where $\widetilde \nabla g(x^{k+1}) \in \partial g(x^{k+1})$.
\end{lemma}
\begin{proof}
    From the minimization step in \eqref{eq:update} we have that
\begin{equation}
    g(x^{k+1}) + \gamma_k \phi(\tfrac{1}{\gamma_k}(x^{k+1}-y^k)) \leq g(x^k) + \gamma_k \phi(\tfrac{1}{\gamma_k}(x^k-y^k)).
\end{equation}
Using \cref{thm:prox_thm} along with the fact that $\phi$ is even and $\nabla \phi^*$ is odd yields
\begin{equation*}
    g(x^{k+1}) + \gamma_k \phi(\nabla \phi^*(\widetilde \nabla g(x^{k+1}))) \leq g(x^k) + \gamma_k \phi(\nabla \phi^*(d^k)),
\end{equation*}
for $\widetilde \nabla g(x^{k+1}) \in \partial g(x^{k+1})$. Applying now the equality case of the Fenchel--Young inequality for the terms $\gamma_k \phi(\nabla \phi^*(\widetilde \nabla g(x^{k+1})))$ and $\gamma_k \phi(\nabla \phi^*(d^k))$ we obtain
\begin{multline} \label{eq:g_bound}
    g(x^{k+1}) + \gamma_k \langle \nabla \phi^*(\widetilde \nabla g(x^{k+1})),\widetilde \nabla g(x^{k+1}) \rangle - \gamma_k\phi^*(\widetilde \nabla g(x^{k+1})) \\ \leq g(x^k) + \gamma_k \langle \nabla \phi^*(d^k),d^k \rangle - \gamma_k\phi^*(d^k).
\end{multline}
On the other hand, using the Euclidean descent lemma for $f$ between points $x^k$ and $x^{k+1}$
\begin{align} \nonumber
    f(x^{k+1}) 
    &\leq f(x^k) - \gamma_k \langle \nabla f(x^k),\nabla \phi^*(d^k) +\nabla \phi^*(\widetilde \nabla g(x^{k+1})) \rangle + 2\gamma_k^2 D^2 L
    \\
    &\leq f(x^k) + 2\gamma_k D\|\lambda^k\| - \gamma_k \langle d^k,\nabla \phi^*(d^k)+\nabla \phi^*(\widetilde \nabla g(x^{k+1})) \rangle  + 2\gamma_k^2 D^2 L, \label{eq:f_bound}
\end{align}
where in the first inequality we used \cref{thm:prox_thm} and in the second inequality we added and subtracted $d^k$ in the inner product term, applied Cauchy--Schwarz and moreover used the bound $\|\nabla \phi^*(y)\| \leq D$ for all $y \in E$.

Summing \eqref{eq:g_bound} and \eqref{eq:f_bound} we get
\begin{multline*} 
    F(x^{k+1}) \leq F(x^k) + 2\gamma_k D\|\lambda^k\| + 2\gamma_k^2 D^2L + \gamma_k\phi^*(\widetilde \nabla g(x^{k+1})) - \gamma_k\phi^*(d^k) \\
    -\gamma_k \langle \nabla \phi^*(\widetilde \nabla g(x^{k+1})),\widetilde \nabla g(x^{k+1}) \rangle - \gamma_k \langle d^k,\nabla \phi^*(\widetilde \nabla g(x^{k+1})) \rangle,
\end{multline*}
which is the claimed result.
\end{proof}

\begin{lemma} \label{thm:comp_desc2}
    Let $\{x^k\}_{k \in \bN}$ be the sequence of iterates generated by \eqref{eq:update}. Then, the following inequality holds
    \begin{equation} \label{eq:c2_bound}
        F(x^{k+1}) \leq F(x^k) + 4\gamma_k D\|\lambda^k\| + 6\gamma_k^2 D^2L - \gamma_k D_{\phi^*}(\nabla f(x^{k+1}), -\widetilde \nabla g(x^{k+1})).
    \end{equation}
    where $\widetilde \nabla g(x^{k+1}) \in \partial g(x^{k+1})$.
\end{lemma}
\begin{proof}
    We start from \cref{thm:comp_desc1} and aim to bound the inner product terms.
    \begin{align*}
        - \langle d^k,\nabla \phi^*(\widetilde \nabla g(x^{k+1})) \rangle 
        &= \langle \nabla f(x^k)- d^k,\nabla \phi^*(\widetilde \nabla g(x^{k+1})) \rangle - \langle \nabla f(x^k),\nabla \phi^*(\widetilde \nabla g(x^{k+1})) \rangle
        \\
        &\leq  D\|\lambda^k\| + D\|\nabla f(x^{k+1})-\nabla f(x^k)\|  -  \langle \nabla f(x^{k+1}),\nabla \phi^*(\widetilde \nabla g(x^{k+1})) \rangle
        \\
        &\leq  D\|\lambda^k\| + 2\gamma_kD^2L -  \langle \nabla f(x^{k+1}),\nabla \phi^*(\widetilde \nabla g(x^{k+1})) \rangle,
    \end{align*}
    where the equality follows by adding and subtracting $\nabla f(x^k)$, the first inequality by applying the Cauchy--Schwarz inequality, using the fact that $\|\nabla \phi^*(y)\| \leq D$ and adding and subtracting $\nabla f(x^{k+1})$. The final inequality inequality follows by using the Lipschitz continuity of $\nabla f$ in the form $\|\nabla f(x^{k+1})-\nabla f(x^k)\| \leq L \|x^{k+1}-x^k\|$ and then the bound from \cref{thm:prox_thm}. Therefore,
    \[
        - \gamma_k \langle d^k,\nabla \phi^*(\widetilde \nabla g(x^{k+1})) \rangle \leq  \gamma_k D\|\lambda^k\| + 2\gamma_k^2D^2L - \gamma_k \langle \nabla f(x^{k+1}),\nabla \phi^*(\widetilde \nabla g(x^{k+1})) \rangle.
    \]
Further, from \cref{eq:rev_ineq} we have that
\begin{align*}
    -\gamma_k \phi^*(d^k) &\leq -\gamma_k \phi^*(\nabla f(x^k)) + \gamma_k D\|\nabla f(x^k)-d^k\|
    \\
    &= -\gamma_k \phi^*(\nabla f(x^{k+1})) + \gamma_k (\phi^*(\nabla f(x^{k+1})) -  \phi^*(\nabla f(x^k))) + \gamma_k D\|\lambda^k\| \\
    &\leq -\gamma_k \phi^*(\nabla f(x^{k+1})) + \gamma_k D\|\nabla f(x^{k+1})-\nabla f(x^{k})\| + \gamma_k D\|\lambda^k\|
    \\
    &\leq -\gamma_k \phi^*(\nabla f(x^{k+1})) + 2\gamma_k^2 D^2L + \gamma_k D \|\lambda^k\|,
\end{align*}
where we have added and subtracted $\phi^*(\nabla f(x^{k+1}))$, used the fact that $\phi^*$ is $D$-Lipschitz continuous, $\|\nabla f(x^{k+1})-\nabla f(x^k)\| \leq L \|x^{k+1}-x^k\|$ and applied the bound from \cref{thm:prox_thm}. Putting everything back in \eqref{eq:main_ineq_bregman} we get
\begin{multline*}
    F(x^{k+1}) \leq F(x^k) + 4\gamma_k D\|\lambda^k\| + 6\gamma_k^2 D^2L + \gamma_k\phi^*(\widetilde \nabla g(x^{k+1})) - \gamma_k\phi^*(\nabla f(x^{k+1}))
    \\
    -\gamma_k \langle \nabla \phi^*(\widetilde \nabla g(x^{k+1})),\widetilde \nabla g(x^{k+1}) +  \nabla f(x^{k+1})\rangle.
\end{multline*}
Now note that 
\begin{equation*}
    \gamma_k\phi^*(\widetilde \nabla g(x^{k+1})) - \gamma_k\phi^*(\nabla f(x^{k+1})) - \gamma_k \langle \nabla \phi^*(\widetilde \nabla g(x^{k+1})),\widetilde \nabla g(x^{k+1}) +  \nabla f(x^{k+1})\rangle 
\end{equation*}
equals $-\gamma_k D_{\phi^*}(\nabla f(x^{k+1}), -\widetilde \nabla g(x^{k+1}))$.

Substituting this result in the inequality above we obtain the claimed result.
\end{proof}

In order to tackle the heavy-tailed noise we follow the approach of \cite{hubler2024gradient} and thus require the following lemma.

\begin{lemma}[{\cite[Lemma 10]{hubler2024gradient}}] \label{thm:p_bound}
    Let $p \in [1,2]$ and $X_1,\dots, X_n \in E$ be a martingale difference sequence, i.e., $\expec[X_j\mid X_{j-1}, \dots, X_1] = 0$ a.s.\ for all $j = 1,\dots,n$ satisfying
    \begin{equation}
        \expec[\|X_j\|^p] < +\infty \qquad \text{for all } j = 1,\dots, n.
    \end{equation}
    Define $S_n = \sum_{j=1}^n X_j$, then
    \begin{equation*}
        \expec[\|S_n\|^p] \leq 2\sum_{j=1}^n \expec [\|X_j\|^p].
    \end{equation*}
\end{lemma}
We remark that the proof of \cref{thm:p_bound} was given in \cite{hubler2024gradient} for $E = \bR^n$. Nevertheless, it has been proven also for separable Hilbert spaces, see for example \cite{pinelis2022multidimensional}. Note that for the remainder of this subsection we have $\gamma_k = \gamma$ and $\alpha_k = \alpha$, i.e., constant stepsize and momentum parameters.
\begin{lemma} \label{thm:stoch_bound}
    Let \cref{assum:stoch,assum:f_lipschitz} hold true and let $d^k = \alpha \nabla f(x^k, \xi^k) + (1-\alpha)d^{k-1}$ with $d^0=\nabla f(x^0,\xi^0)$. Denote $\lambda^k = d^k - \nabla f(x^k)$. Then, for $0 < \alpha \leq 1$, we have the following estimate
    \begin{equation}
        \expec[\|\lambda^k\|] \leq (1-\alpha)^k \sigma + \frac{2\gamma L D}{\alpha} + 2\alpha \sigma \left (\frac{1}{1-(1-\alpha)^p}\right)^{1/p}.
    \end{equation}
    Therefore, choosing $\alpha = (K+1)^{-1/2}$ and $\gamma = \bar \gamma (K+1)^{-3/4}$ for some $\bar \gamma > 0$ we have for $K \geq 3$
    \begin{equation}
        \frac{D}{K+1}\sum_{k=0}^K \expec[\|\lambda^k\|] \leq \frac{D\sigma}{(K+1)^{1/2}} + \frac{2\bar\gamma LD^2}{(K+1)^{1/4}} +  \frac{2^{(p+1)/p}\sigma D (K+1)^{(2-p)/4p}}{p^{1/p}(K+1)^{1/4}}
    \end{equation}
\end{lemma}
\begin{proof}
   Observe that
\begin{equation*}
    \lambda^k = (1-\alpha)\lambda^{k-1} + \alpha (\nabla f(x^k,\xi^k)-\nabla f(x^k)) + (1-\alpha)(\nabla f(x^{k-1})-\nabla f(x^k)),
\end{equation*}
which readily implies
\begin{multline*}
     \lambda^{k} = (1-\alpha)^{k}(d^0-\nabla f(x^0)) + \sum_{i=0}^{k-1}(1-\alpha)^{k-i}(\nabla f(x^i)-\nabla f(x^{i+1}))  \\
     + \alpha \sum_{i=1}^k (1-\alpha)^{k-i}(\nabla f(x^i, \xi^i)-\nabla f(x^i)).
\end{multline*}
Using the triangle inequality we obtain:
\begin{multline} \label{eq:est_upper}
    \|\lambda^k\| \leq (1-\alpha)^{k}\|d^0-\nabla f(x^0)\| + \sum_{i=0}^{k-1}(1-\alpha)^{k-i}\|\nabla f(x^i)-\nabla f(x^{i+1})\|
    \\
     + \alpha  \left\|\sum_{i=1}^k (1-\alpha)^{k-i}(\nabla f(x^i, \xi^i)-\nabla f(x^i)) \right\|.
\end{multline}
We bound the second term in a standard way using the $L$-Lipschitz continuity of $\nabla f$, 
$$\sum_{i=0}^{k-1}(1-\alpha)^{k-i}\|\nabla f(x^i)-\nabla f(x^{i+1})\| \leq L\sum_{i=0}^{k-1}(1-\alpha)^{k-i}\|x^i-x^{i+1}\| \leq 2\gamma L D \sum_{i=0}^{k-1}(1-\alpha)^{k-i},
$$
where the second inequality follows by the bound in \cref{thm:prox_thm}. Regarding the last term we follow the analysis in \cite{hubler2024gradient} and obtain
\[
    \expec \left [\left\|\sum_{i=1}^k (1-\alpha)^{k-i}(\nabla f(x^i, \xi^i)-\nabla f(x^i)) \right\|\right] \leq 
    \expec \left [ \left \|\sum_{i=1}^k (1-\alpha)^{k-i}(\nabla f(x^i, \xi^i)-\nabla f(x^i))\right \|^p \right]^{1/p}
\]
using Jensen's inequality. We now use \cref{thm:p_bound} with $X_i = (1-\alpha)^{k-i}(\nabla f(x^i,\xi^i) - \nabla f(x^i))$, utilizing \cref{assum:stoch}, and obtain 
\begin{align*}
    \expec \left [ \left \|\sum_{i=1}^k (1-\alpha)^{k-i}(\nabla f(x^i, \xi^i)-\nabla f(x^i))\right \|^p \right]^{1/p} &\leq \left ( 2\sum_{i=1}^k (1-\alpha)^{p(k-i)} \sigma^p \right )^{1/p} \\ &\leq 2 \sigma \left ( \sum_{i=1}^k (1-\alpha)^{p(k-i)} \right )^{1/p},
\end{align*}
where we have moreover used \cref{assum:stoch} in the form of $\expec[\|\nabla f(x,\xi)- \nabla f(x)\|^p] \leq \sigma^p$. Bounding the series on the r.h.s.\ and multiplying both sides by $\alpha$, we thus reach the following bound
\begin{equation}
    \expec \left [ \alpha  \left\|\sum_{i=1}^k (1-\alpha)^{k-i}(\nabla f(x^i, \xi^i)-\nabla f(x^i)) \right\|\right] \leq 2\alpha \sigma \left (\frac{1}{1-(1-\alpha)^p}\right)^{1/p}.
\end{equation}
    Furthermore, we have that $\expec[\|d^0-\nabla f(x^0)\|]=\expec[\|\nabla f(x^0,\xi^0)-\nabla f(x^0)\|] \leq \expec[\|\nabla f(x^0,\xi^0)-\nabla f(x^0)\|^p]^{1/p} \leq \sigma$. Putting everything together,
    \begin{equation*}
        \expec[\|\lambda^k\|] \leq (1-\alpha)^k \sigma + \frac{2\gamma L D}{\alpha} + 2\alpha \sigma \left (\frac{1}{1-(1-\alpha)^p}\right)^{1/p}.
    \end{equation*}
    Regarding the second claim we first upper bound $\sum_{k=0}^K \expec[\|\lambda^k\|]$. We begin with the simple bound 
    \[
        \sum_{k=0}^K(1-\alpha)^k \sigma \leq \tfrac{\sigma}{\alpha} = \sigma (K+1)^{1/2}.
    \]
    Further, we have $\sum_{k=0}^K \tfrac{2\gamma L D}{\alpha} = 2\bar \gamma LD (K+1)^{3/4}.$ Finally, from the Mean Value Theorem for $t^p$ in the interval $[1-\alpha,1]$ we have for some $\xi \in (1-\alpha, 1)$ that $p\xi^{p-1} = \frac{1-(1-\alpha)^p}{\alpha}$. Since $\alpha = (K+1)^{-1/2}$, for $K \geq 3$ we have $\alpha \leq 1/2$ and thus $\xi > 1/2$. Therefore, since $p\in(1,2]$, $1-(1-\alpha)^p \geq p\alpha/2 $ and hence $(\tfrac{1}{1-(1-\alpha)^p})^{1/p} \leq (\tfrac{2}{p\alpha})^{1/p} = (K+1)^{1/2p}(\tfrac{2}{p})^{1/p}$. This implies $\sum_{k=0}^K 2\alpha \sigma (\tfrac{1}{1-(1-\alpha)^p})^{1/p} \leq 2^{(p+1)/p}\sigma (K+1)^{(p+1)/2p}/p^{1/p}$ and thus
    \begin{equation*}
        \frac{D}{K+1}\sum_{k=0}^K 2\alpha \sigma (\tfrac{1}{1-(1-\alpha)^p})^{1/p}\leq \frac{2^{(p+1)/p}\sigma D (K+1)^{(2-p)/4p}}{p^{1/p}(K+1)^{1/4}}.
    \end{equation*}
    Putting everything together leads to the claimed result.
\end{proof}

\subsection{Proof of \texorpdfstring{\cref{thm:main_res}}{Theorem 4.3}}

\begin{proof}  
    We start from \cref{thm:comp_desc2} and take expectation with respect to $\xi^k$
    \begin{equation*}
        \expec_{\xi^k}[F(x^{k+1})] \leq F(x^k) + 4\gamma D\expec_{\xi^k}[\|\lambda^k\|] + 6\gamma^2 D^2L - \expec_{\xi^k}[\gamma D_{\phi^*}(\nabla f(x^{k+1}), -\widetilde \nabla g(x^{k+1}))].
    \end{equation*}
    Taking total expectation and summing up the inequality from $k=0$ to $k=K$ leads us to
    \begin{align*}
        \gamma \sum_{k=0}^K \expec[D_{\phi^*}(\nabla f(x^{k+1}), -\widetilde \nabla g(x^{k+1}))] \leq F(x^0)-F_\star + 4 \gamma D \sum_{k=0}^K \expec[\|\lambda^k\|] + 6 L\gamma^2 D^2 (K + 1),
    \end{align*}
    where we also used $F(x) \geq F_\star$ for all $x \in E$. 
    Multiplying with $\tfrac{1}{\gamma(K+1)}$ and using the bound from \cref{thm:stoch_bound} we get
    \[
        \frac{1}{K+1}\sum_{k=0}^K\expec[D_{\phi^*}(\nabla f(x^{k+1}), -\widetilde \nabla g(x^{k+1}))]  \leq A
    \]
    where $A$ equals
    \[
        \frac{F(x^0)-F_\star}{\bar \gamma (K+1)^{1/4}} + \frac{6L\bar \gamma D^2}{(K+1)^{3/4}} +\frac{4D\sigma}{(K+1)^{1/2}} + \frac{8\bar\gamma LD^2}{(K+1)^{1/4}} +  \frac{2^{(p+1)/p}4\sigma D (K+1)^{(2-p)/4p}}{p^{1/p}(K+1)^{1/4}},
    \]
    which is the claimed result.
\end{proof}

\section{Analysis for the STORM estimator}
 \label{app:storm_proof}

We now proceed to bound the error $\lambda^k = d^k - \nabla f(x^k)$ when using the STORM estimator from \cite{cutkosky2019momentum}, i.e. 
\begin{equation} \label{eq:storm}
    d^k = (1-\alpha_{k})d^{k-1} + \alpha_k \nabla f(x^k, \xi^k) + (1-\alpha_k) (\nabla f(x^k,\xi^k) - \nabla f(x^{k-1},\xi^k)).
\end{equation}
Adapting the analysis to our setting, we have the following intermediate result.
\begin{lemma} \label{thm:lemma_bound_diff_storm}
    Let \cref{assum:stoch,assum:f_lipschitz,assum:fxi_lipschitz} hold true with $p=2$ and let $d^k$ be updated as in \eqref{eq:storm} with $d^0=\nabla f(x^0,\xi^0)$. Denote $\lambda^k = d^k - \nabla f(x^k)$. Then, we have the following estimate for $k \geq 1$
    \begin{equation*}
        \expec_{\xi^k}[\|\lambda^k\|^2] \leq (1-\alpha_k)^2\|\lambda^{k-1}\|^2 + 2\alpha_k^2 \sigma^2 + 2(1-\alpha_k)^2 \gamma_{k-1}^2 L^2 \|\nabla \phi^*(d^{k-1})+\nabla \phi^*(\widetilde \nabla g(x^k))\|^2,
    \end{equation*}
    where $\widetilde \nabla g(x^{k+1}) \in \partial g(x^{k+1})$. For $k = 0$, $\expec_{\xi^0}[\|\lambda^0\|]^2 \leq \sigma^2$.
\end{lemma}
\begin{proof}
    Observe that (see also \cite[p.\,3]{cutkosky2019momentum})
    \begin{multline*}
        \lambda^k = (1-\alpha_k)\lambda^{k-1} + \alpha_k (\nabla f(x^k, \xi^k)-\nabla f(x^k)) \\ + (1-\alpha_k)(\nabla f(x^k,\xi^k) - \nabla f(x^{k-1},\xi^k)-(\nabla f(x^k)-\nabla f(x^{k-1}))).
    \end{multline*}
    and using the Pythagorean theorem we have
    \begin{align*}
        \|\lambda^k\|^2 
        &= (1-\alpha_k)^2\|\lambda^{k-1}\|^2 + \alpha_k^2\|\nabla f(x^k, \xi^k)-\nabla f(x^k)\|^2 
        \\
        &+ (1-\alpha_k)^2\|\nabla f(x^k,\xi^k) - \nabla f(x^{k-1},\xi^k)-(\nabla f(x^k)-\nabla f(x^{k-1}))\|^2
        \\
        &+\tilde{\alpha}_k \langle \nabla f(x^k,\xi^k) - \nabla f(x^{k-1},\xi^k)-(\nabla f(x^k)-\nabla f(x^{k-1})),\nabla f(x^k, \xi^k)-\nabla f(x^k) 
        \rangle
        \\
        &+\tilde{\alpha}_k \langle \lambda^{k-1},\nabla f(x^k, \xi^k)-\nabla f(x^k) \rangle 
        \\
        &+ 2(1-\alpha_k)^2 \langle \lambda^{k-1},\nabla f(x^k,\xi^k)- \nabla f(x^{k-1},\xi^k)-(\nabla f(x^k)-\nabla f(x^{k-1}))\rangle 
    \end{align*}
    where $\tilde{\alpha}_k = 2(1-\alpha_k)\alpha_k$.
    
    Now, by using Young's inequality for the fourth term, taking expectation with respect to $\xi^k$ and using \cite[Lemma 3]{cutkosky2019momentum}, i.e., that the expectation of the final two terms is $0$ we get
    \begin{multline*}
        \expec_{\xi^k}[\|\lambda^k\|^2] \leq (1-\alpha_k)^2\|\lambda^{k-1}\|^2 + 2\alpha_k^2 \expec_{\xi^k}[\|\nabla f(x^k, \xi^k)-\nabla f(x^k)\|^2] 
        \\
        + 2(1-\alpha_k)^2\expec_{\xi^k}[\|\nabla f(x^k,\xi^k) - \nabla f(x^{k-1},\xi^k)-(\nabla f(x^k)-\nabla f(x^{k-1}))\|^2]
    \end{multline*}
    For the last term in this inequality, we moreover have
    \begin{align*}
        \expec_{\xi^k}[\|\nabla f(x^k,\xi^k) &- \nabla f(x^{k-1},\xi^k)-(\nabla f(x^k)-\nabla f(x^{k-1}))\|^2] 
        \\
        &= \expec_{\xi^k}[\|\nabla f(x^k,\xi^k) - \nabla f(x^{k-1},\xi^k)\|^2] + \|\nabla f(x^k)-\nabla f(x^{k-1})\|^2
        \\
        &\mspace{100mu} -2\expec_{\xi^k}[\langle \nabla f(x^k,\xi^k) - \nabla f(x^{k-1},\xi^k),\nabla f(x^k)-\nabla f(x^{k-1}) \rangle]
        \\
        &\leq \expec_{\xi^k}[\|\nabla f(x^k,\xi^k) - \nabla f(x^{k-1},\xi^k)\|^2],
    \end{align*}
    where we have used the fact that the stochastic estimator is unbiased. Since $\nabla f(\cdot, \xi^k)$ is Lipschitz continuous, the last quantity in the inequality above can be upper bounded by $L^2 \|x^{k}-x^{k-1}\|^2$. Therefore, we have that
    \begin{equation*}
        \expec_{\xi^k}[\|\lambda^k\|^2] \leq (1-\alpha_k)^2\|\lambda^{k-1}\|^2 + 2\alpha_k^2 \sigma^2 + 2(1-\alpha_k)^2 \gamma_{k-1}^2 L^2 \|\nabla \phi^*(d^{k-1})+\nabla \phi^*(\widetilde \nabla g(x^k))\|^2,
    \end{equation*}
    where we also used \Cref{thm:prox_thm}.
\end{proof}
\subsection{Proof of \texorpdfstring{\cref{thm:conv_storm}}{Theorem 4.5}}
\begin{proof}
By Young's inequality, we have that
\[
    4\gamma_k D \|\lambda^k\| \leq \tfrac{L}{2\gamma_k^{1/2}} (4\gamma_k D)^2 + \tfrac{\gamma_k^{1/2}}{2L}\|\lambda^k\|^2 = 8\gamma_k^{3/2}D^2L + \tfrac{\gamma_k^{1/2}}{2L}\|\lambda^k\|^2
\]
which we substitute into \eqref{eq:c2_bound} to obtain
 \begin{equation} \label{eq:storm_est_finite_dom}
    F(x^{k+1}) \leq F(x^k) + 8\gamma_k^{3/2} D^2 L + \tfrac{\gamma_k^{1/2}}{2L}\|\lambda^k\|^2 + 6\gamma_k^2 D^2L - \gamma_k D_{\phi^*}(\nabla f(x^{k+1}), -\widetilde \nabla g(x^{k+1})).
\end{equation}
Taking total expectation in \cref{thm:lemma_bound_diff_storm}, and noting that $\|\nabla \phi^*(d^{k-1})+\nabla \phi^*(\widetilde \nabla g(x^{k+1}))\|^2 \leq 4D^2$ we have that
\[
    \tfrac{1}{L}(\tfrac{1}{\gamma_k^{1/2}}\expec[\|\lambda^k\|^2]-\tfrac{1}{\gamma_{k-1}^{1/2}}\expec[\|\lambda^{k-1}\|^2]) \leq \tfrac{1}{L}(\tfrac{(1-\alpha_k)^2}{\gamma_k^{1/2}}-\tfrac{1}{\gamma_{k-1}^{1/2}})\expec[\|\lambda^{k-1}\|^2]   +\tfrac{2 \alpha_k^2}{\gamma_k^{1/2}L}\sigma^2 + \tfrac{8(1-\alpha_k)^2 LD^2 \gamma_{k-1}^2}{\gamma_k^{1/2}}.
\]
Now note that since $\gamma_k = (k+1)^{-2/3}$ and $\alpha_k = \gamma_k$ we have $\tfrac{1}{L}(\tfrac{(1-\alpha_k)^2}{\gamma_k^{1/2}}-\tfrac{1}{\gamma_{k-1}^{1/2}}) \leq -\tfrac{1}{2L}k^{-1/3} = -\tfrac{1}{2L}\gamma_{k-1}^{1/2}$ and $(1-\alpha_k)^2 (\tfrac{\gamma_{k-1}}{\gamma_k})^{1/2} \leq 1$. Therefore,
\begin{equation} \label{eq:storm_est_finite_dom2}
    \tfrac{1}{L}\tfrac{1}{\gamma_k^{1/2}}\expec[\|\lambda^k\|^2] \leq \tfrac{1}{L}(\tfrac{1}{\gamma_{k-1}^{1/2}} -  \tfrac{\gamma_{k-1}^{1/2}}{2})\expec[\|\lambda^{k-1}\|^2]   +\tfrac{2 \alpha_k^2}{\gamma_k^{1/2}L}\sigma^2 + 8LD^2\gamma_{k-1}^{3/2}.
\end{equation}
Taking total expectation in \eqref{eq:storm_est_finite_dom} and summing with \eqref{eq:storm_est_finite_dom2} for $k \geq 1$ we obtain 
\begin{align*}
    \expec[F(x^{k+1})] + \tfrac{1}{L}(\tfrac{1}{\gamma_k^{1/2}}-\tfrac{\gamma_k^{1/2}}{2})\expec[\|\lambda^k\|^2]
    &\leq \expec[F(x^k)] + \tfrac{1}{L}(\tfrac{1}{\gamma_{k-1}^{1/2}}-\tfrac{\gamma_{k-1}^{1/2}}{2})\expec[\|\lambda^{k-1}\|^2] + \tfrac{2 \alpha_k^{3/2}}{L}\sigma^2 \\
    &\mspace{20mu}  + 8LD^2 \gamma_{k-1}^{3/2} -\gamma_k\expec[D_{\phi^*}(\nabla f(x^{k+1}), -\widetilde \nabla g(x^{k+1}))]
    \\
    &\mspace{20mu} + 8\gamma_k^{3/2} D^2 L + 6\gamma_k^2 D^2L.
\end{align*}
Summing up the inequality above from $k=1$ to $k=K$, and using standard integral upper bounds we get
\begin{align*}
    \sum_{k=1}^K\gamma_k\expec[D_{\phi^*}(\nabla f(x^{k+1}), -\widetilde \nabla g(x^{k+1}))] 
    &\leq \expec[F(x^1)-F(x^{K+1})] + \tfrac{1}{2L}\sigma^2 + \tfrac{2}{L}\ln(K+1)\sigma^2 
    \\
    & + 8LD^2(1+\ln(K))  + 8 D^2L \ln(K+1) + 18 LD^2.
\end{align*}
Summing now the above with \eqref{eq:storm_est_finite_dom} for $k=0$ we obtain:
\[
    \sum_{k=0}^K\gamma_k\expec[D_{\phi^*}(\nabla f(x^{k+1}), -\widetilde \nabla g(x^{k+1}))] \leq A
\]
where $A$ equals
\[
     F(x^0)-F_{\star}+
    \tfrac{1}{L}(1+2\ln(K+1))\sigma^2 + (40+16\ln(K+1))LD^2.
\]
Clearly $\gamma_K \leq \gamma_k$ and thus multiplying with $\tfrac{1}{\gamma_K(K+1)} = \tfrac{1}{(K+1)^{1/3}}$ leads to the claimed result.
\end{proof}

\section{Analysis of Polar Express} \label{app:polar_full}

\subsection{Polar Express reference function} \label{app:polar_discussion}
Consider a family of nonlinear preconditioners $(h^*_{\kappa})'(x) = \frac{x}{(\varepsilon^{\kappa} + |x|^{\kappa})^{1/\kappa}}$ with $\kappa \geq 1$. These arise from reference functions of the form $h(x) = \varepsilon \frac12 |x|^2 {}_2F_1(\frac{2}{\kappa}, \frac{1}{\kappa}; 1 + \frac{2}{\kappa}; |x|^\kappa)$, where ${}_2F_1$ is the ordinary hypergeometric function. It can be shown that these reference functions satisfy \Cref{assum:sc} with strong convexity parameter $\mu = \varepsilon$. Note that for $\kappa=1$, the preconditioner from \Cref{example:norm_grad} is recovered while $\kappa=2$ is related to Adagrad in light of \cite[Example 1.5]{oikonomidis2025nonlinearlypreconditionedgradientmethods}.

Although the Polar Express method \cite{amsel2025polar} is nonmonotone and therefore does not satisfy the assumptions of our framework, it approximates $(h_4^*)'$ more closely than the actual sign function, as illustrated in \Cref{fig:pe_approx_app}. In practice, a Frobenius normalization step is performed prior to applying this polynomial to ensure convergence. In the following sections, we show how our analysis can deal with this required input normalization.

\begin{figure}[thbp]
    \centering
    \begin{subfigure}[b]{0.32\textwidth}
        \centering
        \includegraphics[width=\textwidth]{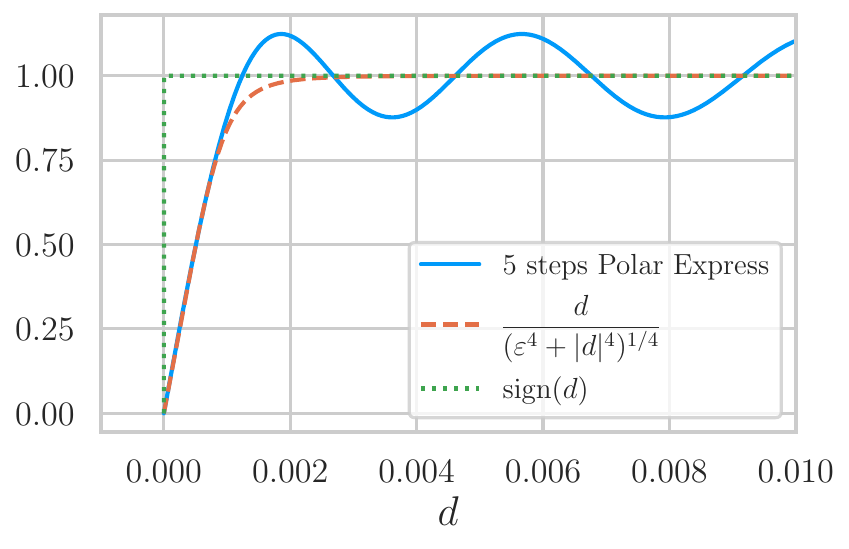}
        \caption{$\varepsilon=10^{-3}$}
    \end{subfigure} \hfill
    \begin{subfigure}[b]{0.32\textwidth}
        \centering
        \includegraphics[width=\textwidth]{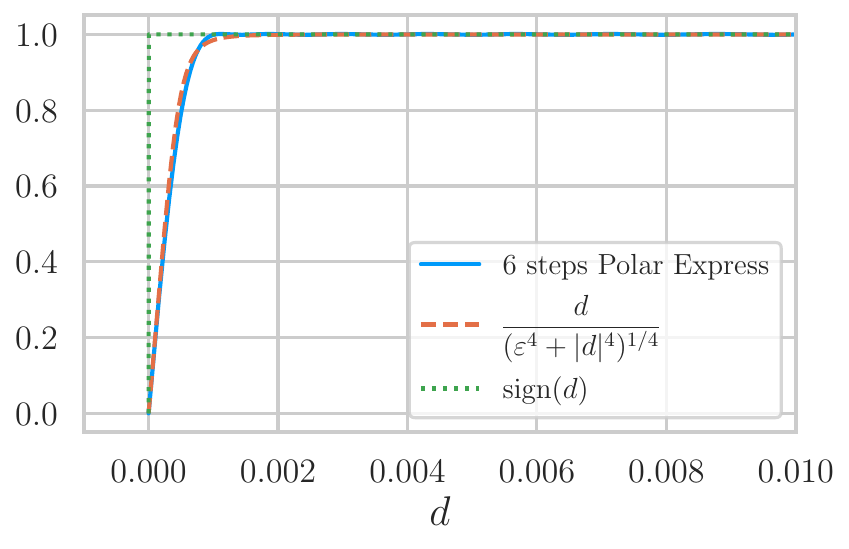}
        \caption{$\varepsilon=5\times 10^{-4}$}
    \end{subfigure} \hfill
    \begin{subfigure}[b]{0.32\textwidth}
        \centering
        \includegraphics[width=\textwidth]{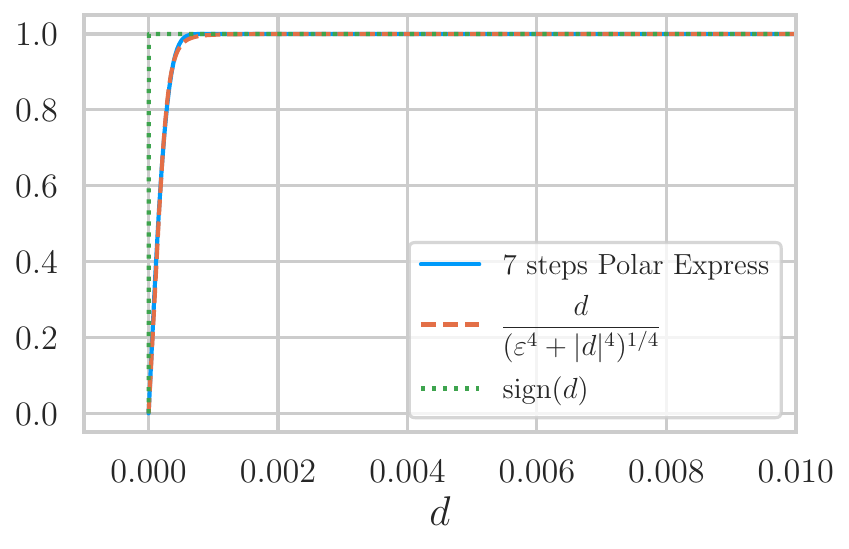}
        \caption{$\varepsilon=3\times 10^{-4}$}
    \end{subfigure}
    \caption{The Polar Express fits more closely to $(h^*_4)'$ than to the actual sign, assuming that the input is normalized to have spectral norm less than or equal to 1.}
    \label{fig:pe_approx_app}
\end{figure}
\subsection{Deterministic setting} \label{app:inexact_update}
In this subsection we show how the discussion above can be incorporated into our analysis. We first consider the vector setting. Therefore, assume that $E = \bR^n$ and $\phi$ is the separable reference function $\phi(x) = \sum_{i=1}^n h_\kappa(x_i)$, thus leading to the update
\begin{equation*}
    x^{k+1} = x^k - \gamma \nabla \phi^* \left (\frac{\nabla f(x^k)}{\|\nabla f(x^k)\|} \right),
\end{equation*}
where we take $\nabla f(x^k) /\|\nabla f(x^k)\| = 0$ if $\nabla f(x^k) = 0$. Assuming that $\nabla f(x^k) \neq 0$ for all $i \in \{0, 1,\dots, k\}$ because otherwise the algorithm would have stopped, we have the following from the Lipschitz smoothness inequality
\begin{align*}
    f(x^{k+1}) \leq f(x^k) - \gamma \left\langle \nabla f(x^k),\nabla \phi^* \left (\frac{\nabla f(x^k)}{\|\nabla f(x^k)\|} \right) \right\rangle + \frac{L\gamma^2 D^2}{2}.
\end{align*}
Note now that for $z_i = \nabla_i f(x^k)/\|\nabla f(x^k)\|$ we have 
\begin{align*}
    \left\langle \nabla f(x^k),\nabla \phi^* \left (\frac{\nabla f(x^k)}{\|\nabla f(x^k)\|} \right) \right\rangle= \|\nabla f(x^k)\|\sum_{i=1}^n \frac{z_i^2}{(\varepsilon^\kappa + |z_i|^\kappa)^{1/\kappa}}.
\end{align*}
and $(\varepsilon^\kappa + |z_i|^\kappa)^{1/\kappa} \leq \varepsilon + |z_i|$. This implies
\begin{align*}
    \left\langle \nabla f(x^k),\nabla \phi^* \left (\frac{\nabla f(x^k)}{\|\nabla f(x^k)\|} \right) \right\rangle
    &\geq \sum_{i=1}^n \frac{|\nabla_i f(x^k)|^2}{\varepsilon\|\nabla f(x^k)\| + |\nabla_i f(x^k)|}.
\end{align*}
Since $\varepsilon\|\nabla f(x^k)\| + |\nabla_i f(x^k)| \leq (1+\varepsilon) \|\nabla f(x^k)\|$, we find that
\[
    \left\langle \nabla f(x^k),\nabla \phi^* \left (\frac{\nabla f(x^k)}{\|\nabla f(x^k)\|} \right) \right\rangle \geq \frac{\|\nabla f(x^k)\|}{1+\varepsilon}.
\]
Choosing thus $\gamma$ small enough (like $(K+1)^{-1/2}$) we obtain standard convergence rates, i.e.,
\begin{equation*}
    \min_{k \in \{0,1,\dots, K\}} \|\nabla f(x^k)\| \leq (1+\varepsilon)\frac{2(f(x^0)-f_{\star})+ LD^2}{2\sqrt{K+1}}
\end{equation*}
Now, assume that $E=\bR^{m\times n}$ and denote $\nabla f(X^k) = U\Diag(s){V}^\top$ for an SVD of $\nabla f(X^k)$. Further, let $\Phi$ be the spectral anisotropic reference function associated to $h$. Then, upon using \Cref{cor:spectral_ref}, we have that
\[
    \left\langle \nabla f(X^k), \nabla \Phi^* \left(\frac{\nabla f(X^k)}{\|\nabla f(X^k)\|}\right) \right\rangle = \left\langle U\Diag(s){V}^\top, U \Diag({h^*}'(z_1), \ldots, {h^*}'(z_q)){V}^\top\right\rangle
\]
where $q = \min(m,n)$ and $z_i = s_i / \|\nabla f(X^k)\|$. Now, using real orthogonality of $U$ and $V$, and the cyclic property of the trace, we obtain
\[
    \left\langle \nabla f(X^k), \nabla \Phi^* \left(\frac{\nabla f(X^k)}{\|\nabla f(X^k)\|}\right) \right\rangle = \trace(\Diag(s)^\top \Diag({h^*}'(z_1), \ldots, {h^*}'(z_q))) = \sum_{i=1}^q s_i {h^*}'(z_i).
\]
Then, by the same calculations as in the vector setting, and using $\|\nabla f(X^k)\| = \sqrt{\sum_{i=1}^q s_i^2} = \|s\|$, we conclude that
\[
     \left\langle \nabla f(X^k), \nabla \Phi^* \left(\frac{\nabla f(X^k)}{\|\nabla f(X^k)\|}\right) \right\rangle  \geq \frac{\|\nabla f(X^k)\|}{1 + \varepsilon}.
\]

\subsection{Stochastic setting} \label{app:stoch_inexact}
In this subsection we analyze the convergence of the scheme described in \cref{app:inexact_update}, tailored to the stochastic setting. To that aim we assume that $E = \bR^n$ to ease notation, since the analysis in the matrix setting follows similarly as in \cref{app:inexact_update}. Moreover $g$ is identically zero, $f_\star = \inf_{x \in E}f(x) > -\infty$, $(h_\kappa^*)'(y) = \frac{y}{(\varepsilon^\kappa + |y|^\kappa)^{1/\kappa}}$ and $\phi(x) = \sum_{i=1}^n h_\kappa(x_i)$. We thus consider the following iterate
\begin{equation}
    x^{k+1} = x^k - \gamma \nabla \phi^*\left (\frac{d^k}{\|d^k\|+\epsilon}\right ),
\end{equation}
where $\epsilon > 0$ is chosen similarly to \cite{amsel2025polar}. We highlight the difference between $\varepsilon$ being an intrinsic part of the preconditioner $\nabla \phi^*$ and $\epsilon$ is part of the pre-normalization step that ensures the polynomial iteration does not diverge. In fact, the whole update can be seen as $x^{k+1} = x^k - \gamma (\nabla \phi^* \circ \nabla \varphi^*)(d^k)$ for $\varphi = -\epsilon(\ln(1-\|\cdot\|) + \|\cdot\|)$ similarly to \cref{example:norm_grad}.

We start the proof similarly to our analysis so far by using the Lipschitz smoothness inequality and adding and subtracting $d^k$:
\begin{equation} \label{eq:lipschitz_desc_approx}
    f(x^{k+1}) \leq f(x^k) + \gamma D \|\lambda^k\| - \gamma \left\langle d^k, \nabla \phi^* \left(\frac{d^k}{\|d^k\|+\epsilon} \right)\right\rangle + \frac{L\gamma^2 D^2}{2},
\end{equation}
where the bound follows by the fact that $\|\nabla \phi^*\| \leq D$. We have the following
\[
    \left\langle d^k, \nabla \phi^*\left(\frac{d^k}{\|d^k\|+\epsilon}\right)\right\rangle 
    = \sum_{i=1}^n\frac{ \tfrac{|d_i^k|^2}{(\|d^k\| + \epsilon)}}{(\varepsilon^\kappa + (\tfrac{|d_i^k|}{\|d^k\|+\epsilon})^\kappa)^{1/\kappa}} \geq   \sum_{i=1}^n \frac{ \tfrac{|d_i^k|^2}{(\|d^k\| + \epsilon)}}{\varepsilon + \tfrac{|d_i^k|}{\|d^k\|+\epsilon}} = \sum_{i=1}^n\frac{|d_i^k|^2}{\varepsilon(\|d^k\| + \epsilon) + |d_i^k|}
\]
where the inequality follows by $(\varepsilon^\kappa + (\tfrac{|d_i^k|}{\|d^k\|+\epsilon})^\kappa)^{1/\kappa} \leq \varepsilon + \tfrac{|d_i^k|}{\|d^k\|+\epsilon}$. Then, since  $|d_i^k| \leq \|d^k\|$, we have that
\[
     \left\langle d^k, \nabla \phi^*\left(\frac{d^k}{\|d^k\|+\epsilon}\right)\right\rangle  \geq \sum_{i=1}^n\frac{|d_i^k|^2}{(\varepsilon+1)\|d^k\| + \epsilon \varepsilon} = \frac{\|d^k\|^2}{(\varepsilon+1)\|d^k\| + \epsilon \varepsilon} =: r(\|d^k\|).
\]
The function $r$ defined by $r(t)= \tfrac{t^2}{(\varepsilon + 1)|t|+\epsilon \varepsilon}$ is $1$-Lipschitz continuous and we thus have $r(\|d^k\|) \geq r(\|\nabla f(x^k)\|) - \|d^k-\nabla f(x^k)\|$. Putting this result back into \eqref{eq:lipschitz_desc_approx} we then have
\begin{equation*}
    f(x^{k+1}) \leq f(x^k) + \gamma (D+1)\|\lambda^k\| + \frac{L\gamma^2 D^2}{2} - \gamma r(\|\nabla f(x^k)\|).
\end{equation*}
Therefore, taking total expectation we are led to
\begin{equation*}
    \gamma \expec[r(\|\nabla f(x^k)\|] \leq \expec[f(x^k) - f(x^{k+1})] + \gamma (D+1)\expec[\|\lambda^k\|] + \frac{L\gamma^2 D^2}{2}.
\end{equation*}
Now note that $r$ is a convex function and thus from Jensen's inequality $\expec[r(\|\nabla f(x^k)\|)] \geq r(\expec[\|\nabla f(x^k)\|])$. Summing the inequality from $k=0$ to $K$ and dividing by $\gamma$ we get
\begin{equation*}
    \sum_{k=0}^K r(\expec[\|\nabla f(x^k)\|]) \leq \frac{f(x^0)-f_{\star}}{\gamma} + \frac{L\gamma D^2 (K+1)}{2} + (D+1) \sum_{k=0}^K \expec[\|\lambda^k\|].
\end{equation*}
Now note again that since $r$ is convex, from Jensen's inequality applied to the sum we have
\begin{align*}
    \sum_{k=0}^K r(\expec[\|\nabla f(x^k)\|]) 
    &\geq (K+1) r\left( \sum_{k=0}^K \frac{\expec[\|\nabla f(x^k)\|]}{K+1}\right) \\
     &= \frac{(\sum_{k=0}^K\expec[\|\nabla f(x^k)\|])^2}{(1+\varepsilon)\sum_{k=0}^K\expec[\|\nabla f(x^k)\|] + \epsilon \varepsilon(K+1)}
\end{align*}
Dividing by $K+1$, and using $\gamma = (K+1)^{-3/4}$ and  \cref{thm:stoch_bound} also accounting for the fact that $\|x^{k+1}-x^k\| \leq \gamma D$ in the unconstrained setting, we then obtain
\begin{equation*}
    \frac{1}{K+1}\frac{(\sum_{k=0}^K\expec[\|\nabla f(x^k)\|])^2}{(1+\varepsilon)\sum_{k=0}^K\expec[\|\nabla f(x^k)\|] + \epsilon \varepsilon(K+1)} 
   \leq A
\end{equation*}
where $A$ equals
\[
   \frac{f(x^0)-f_{\star}}{(K+1)^{1/4}} + \frac{LD^2}{2(K+1)^{3/4}} + (D+1)\left[\frac{\sigma}{(K+1)^{1/2}} + \frac{2\bar\gamma LD}{(K+1)^{1/4}} +  \frac{2^{(p+1)/p}\sigma (K+1)^{(2-p)/4p}}{p^{1/p}(K+1)^{1/4}}\right].
\]
This is a quadratic inequality in $t := \sum_{k=0}^K\expec[\|\nabla f(x^k)\|]$ which implies
\begin{equation*}
    t \leq \frac{(K+1)(1+\varepsilon)A + \sqrt{((K+1)(1+\varepsilon)A)^2 + 4A(K+1)^2 \epsilon \varepsilon}}{2}.
\end{equation*}
Using the standard bound $\sqrt{a + b} \leq \sqrt{a} + \sqrt{b}$ and dividing by $K+1$ we get
\begin{equation*}
    \frac{1}{K+1}\sum_{k=0}^K\expec[\|\nabla f(x^k)\|] \leq (1+\varepsilon)A + \sqrt{A \epsilon \varepsilon}.
\end{equation*}
Note that since $A \leq O((K+1)^{-\min\{\tfrac{1}{4},\tfrac{p-1}{2p}\}})$, the result follows for $\epsilon \leq (K+1)^{-\min\{\tfrac{1}{4},\tfrac{p-1}{2p}\}}$.

\section{Experiments} \label{app:experiments}
Although constraints are ubiquitous in machine learning and thus the need for geometry-aware adaptive optimizers is clear, in this section we conduct some experiments that showcase the effect of weight constraints in neural network training. We remark that the scope of this section is purely to indicate the potential benefit of weight constraints and not to compare optimized algorithms. Efficient implementations and tuning strategies of the methods we describe can further be explored by practitioners.
\subsection{Modular arithmetic}
Grokking is a phenomenon in training neural networks where generalization occurs long after achieving near-zero training error \cite{power2022grokking}. This transition has been most often encountered when training for algorithmic tasks such as learning modular division. We train a small transformer architecture with modernizations (RoPE, RMSNorm, SiLU, $\ldots$), following the implementation of \url{https://github.com/atveit/torch_grokking}. The experiments were conducted on an internal CPU cluster, and the results are shown in \cref{fig:grokking}. We observe that grokking is accelerated when incorporating aggressive weight decay in AdamW, which is consistent with the findings in \cite{power2022grokking}. Similarly, enforcing spectral norm constraints speeds up generalization and allows for larger parameter ranges.
\begin{figure}[htbp]
    \centering
    \begin{subfigure}[b]{0.32\textwidth}
    \includegraphics[width=\textwidth]{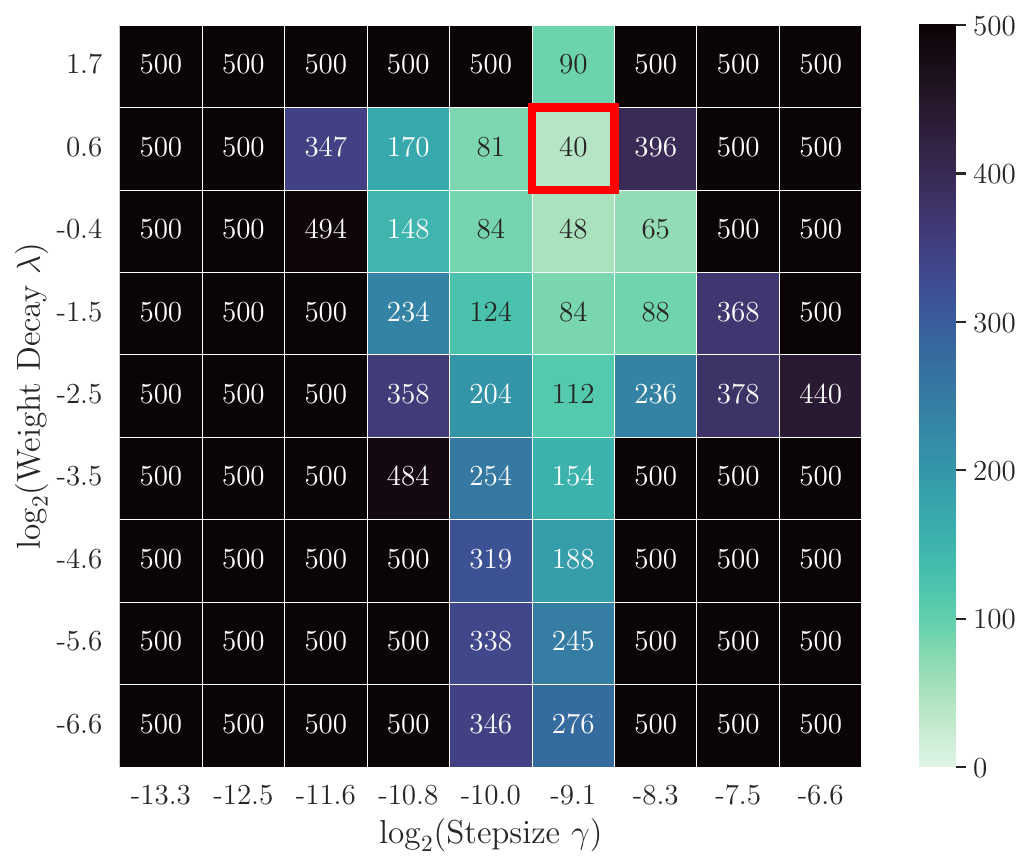}
    \end{subfigure} 
    \hfill 
     \begin{subfigure}[b]{0.32\textwidth}
    \includegraphics[width=\textwidth]{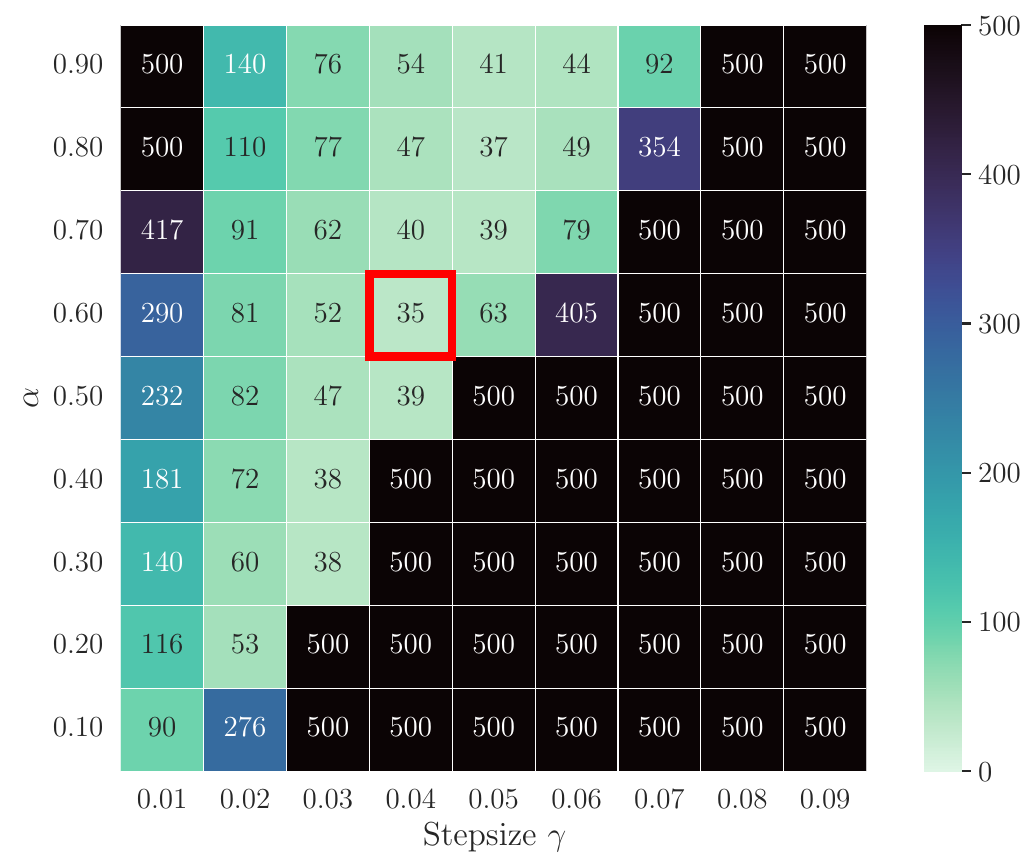}
    \end{subfigure} \hfill
    \begin{subfigure}[b]{0.32\textwidth}
    \includegraphics[width=\textwidth]{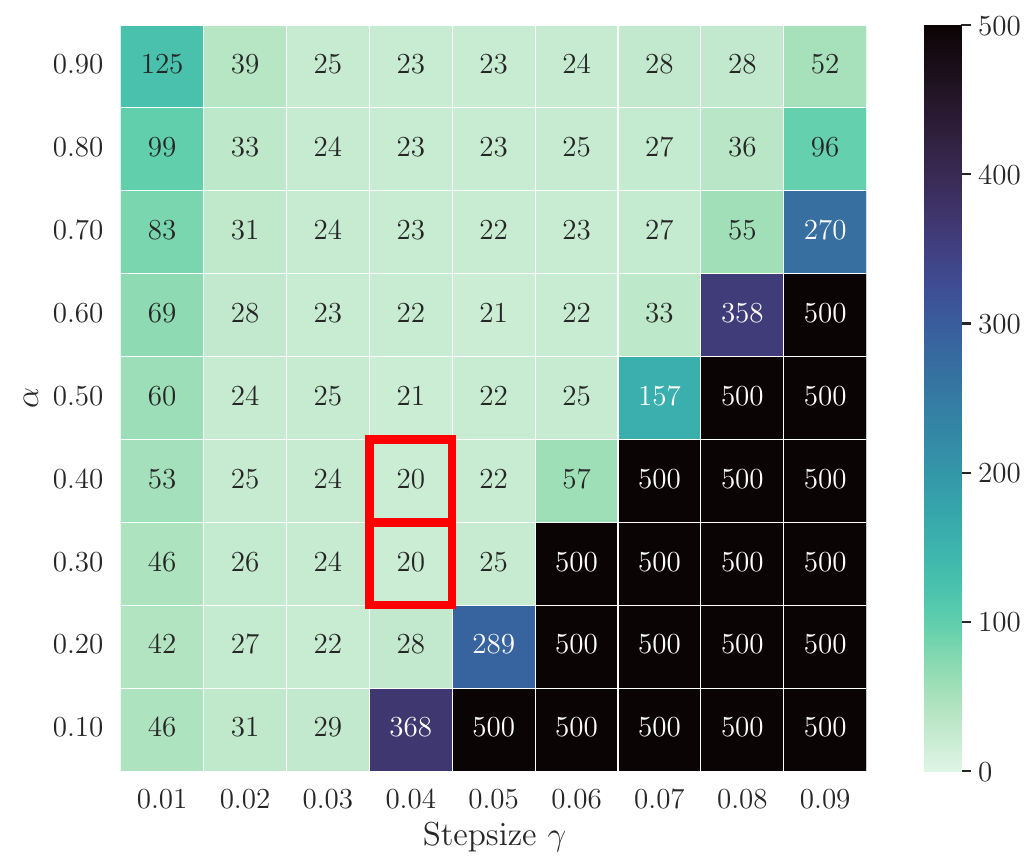} 
    \end{subfigure}
    \caption{Median number of epochs to 95\% validation accuracy for training a transformer on the \texttt{DivMod97} task across 10 different seeds and with a maximum number of epochs set to 500. (left) AdamW (middle) Unconstrained Scion \cite[Recommendation 3.1]{pethick2025trainingdeeplearningmodels} with $\rho_l=1$, $\rho_L=10$ (following the notation of \cite{pethick2025trainingdeeplearningmodels}). (right)  Unconstrained Scion with singular value clipping step on top of the spectral layers. The radius of the weight constraint was set to the same as the radius of the corresponding lmo for the gradient preconditioning. No weight constraint was applied to the embedding layers.} 
    \label{fig:grokking}
\end{figure}

\subsection{CIFAR10}
We adopt the same experimental setup as in \cite{pethick2025trainingdeeplearningmodels} and train a small convolutional network on the CIFAR10 dataset \cite{krizhevsky2009learning}. We use the standard Polyak momentum of \cref{thm:main_res}. These experiments were conducted on an NVIDIA A100 GPU. We compare three optimizers in \Cref{fig:cifar_airbench}: Unconstrained Scion \cite[Recommendation 3.1]{pethick2025trainingdeeplearningmodels}, Unconstrained Scion with singular value clipping on top (projection onto the spectral ball) and Unconstrained Scion with Euclidean projection onto the Frobenius ball as a heuristic that approximates the anisotropic projection without requiring an SVD. The results are averaged over 50 runs. It can be seen that in this setup weight constraints lead to better performance similarly to \cite[Figure 10]{pethick2025trainingdeeplearningmodels}.
\begin{figure}[hbpt]
   \centering
    \begin{subfigure}[b]{0.32\textwidth}
    \includegraphics[width=\textwidth]{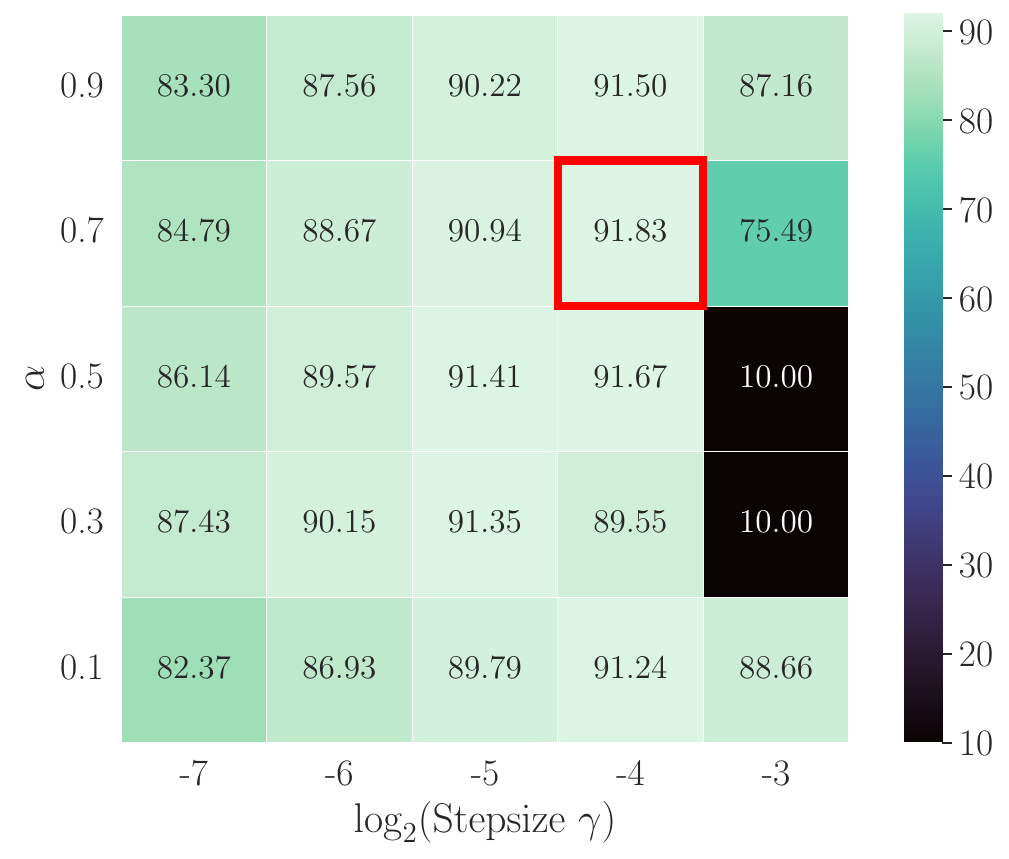}
    \end{subfigure} 
    \hfill 
     \begin{subfigure}[b]{0.32\textwidth}
    \includegraphics[width=\textwidth]{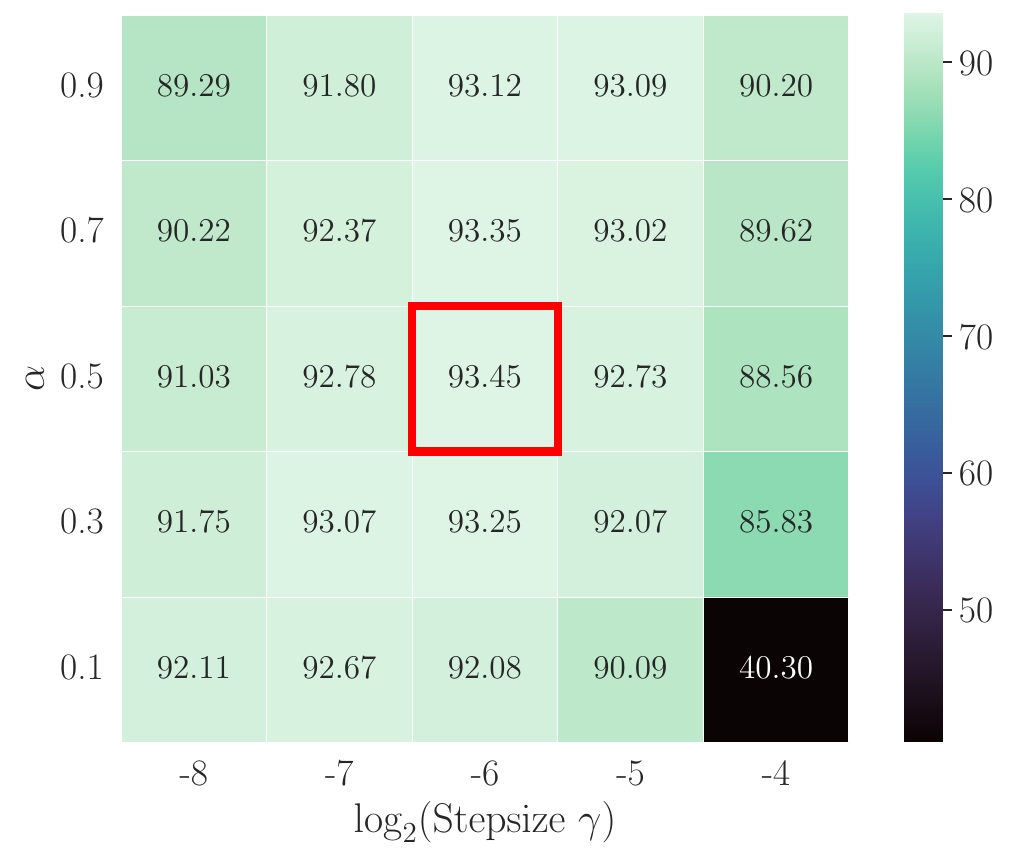}
    \end{subfigure} \hfill
    \begin{subfigure}[b]{0.32\textwidth}
    \includegraphics[width=\textwidth]{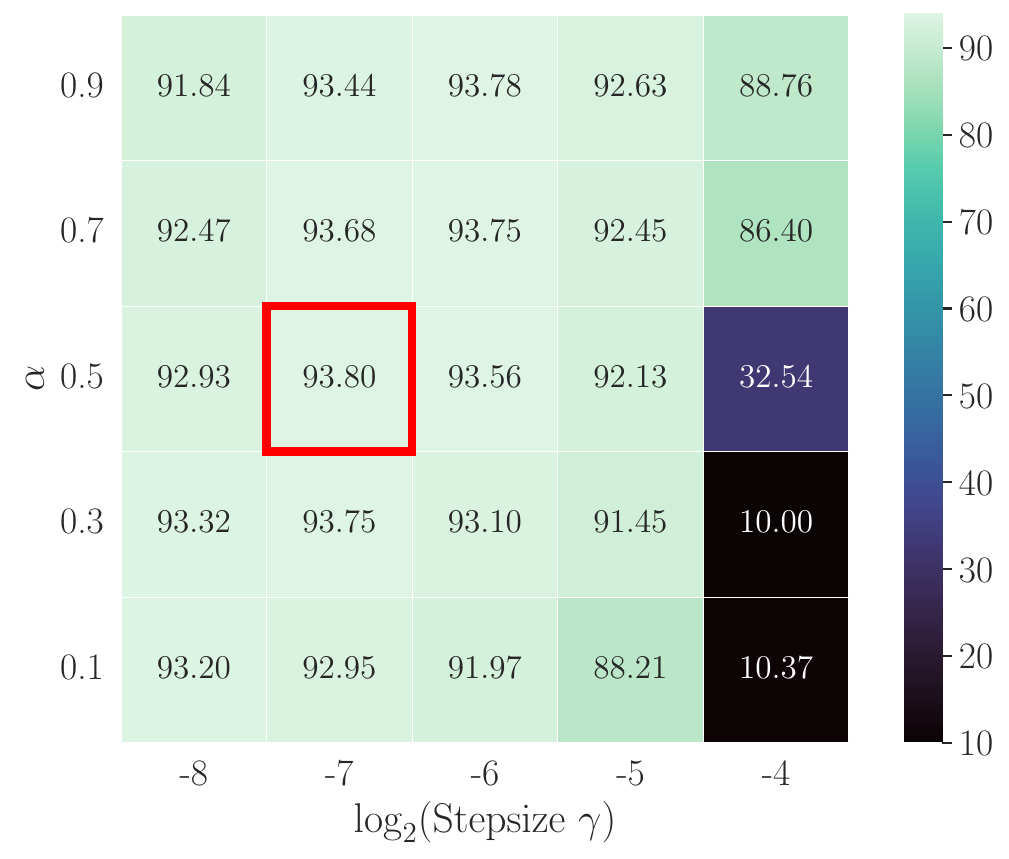} 
    \end{subfigure}
    \caption{Test accuracy for CNN trained on the CIFAR10 dataset. The radii of the lmos in the forward steps were set according to \url{https://github.com/LIONS-EPFL/scion}, i.e., $\rho_l=8, \rho_L=128$ (following the notation of \cite{pethick2025trainingdeeplearningmodels}). (left) Unconstrained Scion (middle) Unconstrained Scion with singular value clipping on top of the spectral layers. The radius of the weight constraint was set to the same radius as the lmo. (right) Unconstrained Scion with Frobenius ball constraint on top of the spectral layers, and the radius of the weight constraint was set to the same radius as the lmo.} 
    \label{fig:cifar_airbench}
\end{figure}

\subsection{NanoGPT}
We build on the code of \cite{pethick2025trainingdeeplearningmodels} and \url{https://github.com/kellerjordan/modded-nanogpt} to train on the NanoGPT benchmark with 124M parameters. Training was performed on 2 NVIDIA H100 GPUs. Instead of the usual 5100 iterations, we train for 10200 iterations so the differences are more pronounced. As a baseline, we consider unconstrained Scion, where we used the same parameters as \cite[Table 7]{pethick2025trainingdeeplearningmodels} and further swept the initial learning rate over $\{0.00018, 0.00022, 0.00026, 0.00032, 0.00036\}$, of which $0.00032$ was the best. We kept the same learning rate schedule with 2900 warmdown iterations. Moreover, we also replaced the matrix sign with the Polar Express \cite{amsel2025polar}, which gave better performance.

For our method, we apply a Frobenius ball constraint to the spectral layers with a radius equal to half of the corresponding lmo. The forward step uses the one from \cref{ex:muon_fwd} and the backward step was computed using the procedure described in \cref{sec:aniso_ball}. Exact SVDs have been used so it is exactly covered by the theory, though it leads to increased computational costs. The best initial learning rate was $0.00018$. The results are shown in \cref{fig:nanogpt_exp}.

\begin{figure}[hbpt]
    \centering
    \includegraphics[width=0.95\linewidth]{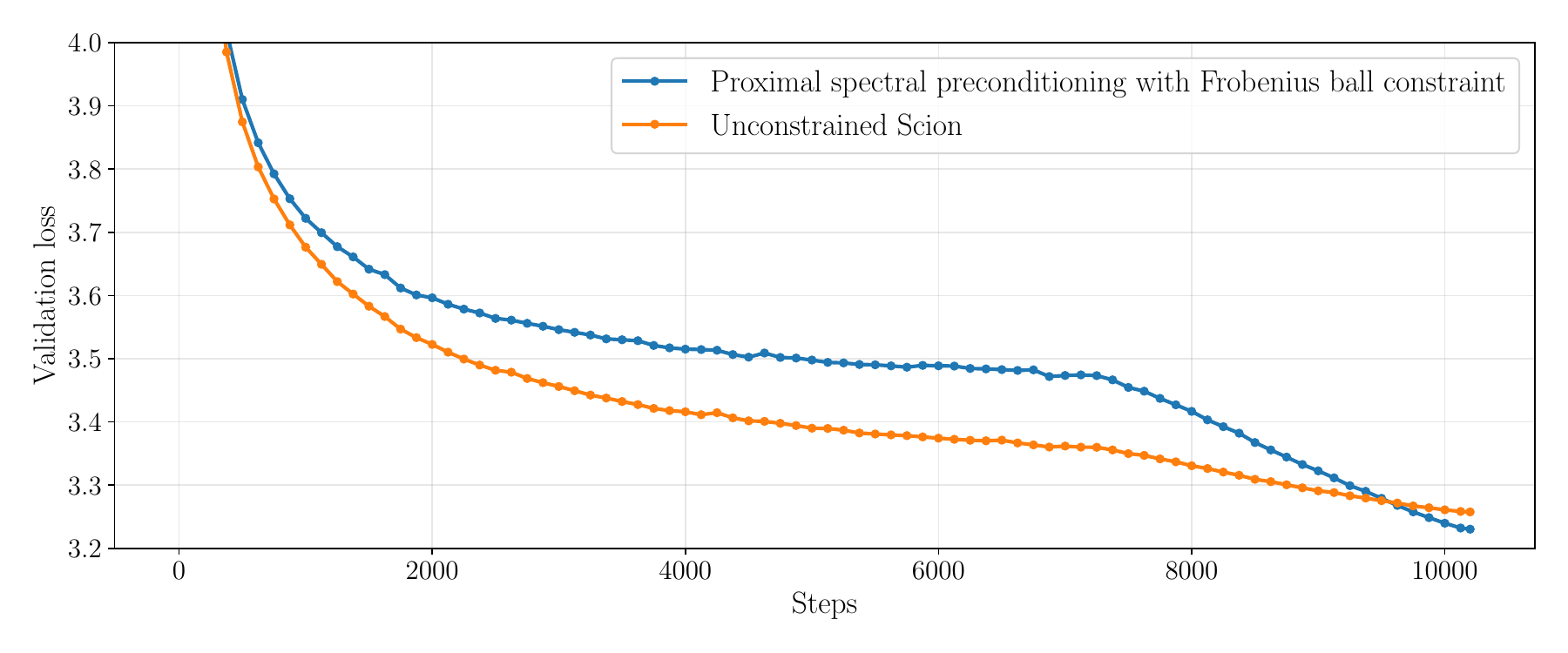}
    \caption{Comparison unconstrained Scion with proximal spectral preconditioning using a Frobenius ball constraint. The behavior is similar to Scion (the constrained version of uScion) that is covered by our theory, see \cref{rem:scion_conn}, as in \cite[Figure 9]{pethick2025trainingdeeplearningmodels}. This is expected since both methods incorporate weight norm constraints.}
    \label{fig:nanogpt_exp}
\end{figure}

\end{document}